\documentclass{article}
\usepackage{amsmath,amsthm,amsfonts,amssymb}
\usepackage{float}
\usepackage{todonotes}
\usepackage[a4paper]{geometry}
\usepackage{changepage}
\usepackage[hidelinks]{hyperref}
\newcommand\R{\mathbb R}
\newcommand\Z{\mathbb Z}
\newcommand\N{\mathbb N}
\newcommand\dist{\operatorname{dist}}
\newcommand\SO{\mathrm{SO}}
\newcommand\calH{\mathcal H}
\newcommand{\Oh}{\mathcal O}
\newcommand\betaeq{\beta_{\mathrm{eq}}}
\newcommand{\sgn}{\operatorname{sgn}}
\newcommand\Wdd{W_\mathrm{2D}}
\newcommand\Wddd{W_\mathrm{3D}}
\newcommand\delamy{b}
\newcommand\betacrit{\beta_\text{crit}}
  \newcommand\larc{\ell_\mathrm{arc}}
\newcommand\calL{\mathcal L}
\newcommand\energy{{\cal E}}

\newcommand\restr[2]{{%
		\left.\kern-\nulldelimiterspace %
		#1 %
		\right|_{#2} %
}}

\newtheorem{theorem}{Theorem}[section]
\newtheorem{proposition}[theorem]{Proposition}
\newtheorem{lemma}[theorem]{Lemma}

\theoremstyle{definition}

\theoremstyle{remark}
\newtheorem{remark}[theorem]{Remark}
\numberwithin{equation}{section}

\newcommand{\LM}[1]{\hbox{\vrule width.2pt \vbox to#1pt{\vfill \hrule width#1pt height.2pt}}}
\newcommand{\LL}{{\mathchoice
{\,\LM7\,}{\,\LM7\,}{\,\LM5\,}{\,\LM{3.35}\,}}}

\begin{document}
	\begin{center}
{\Large
Variational Modeling of Paperboard Delamination 
Under Bending}\\[5mm]
{\today}\\[5mm]
 {Sergio Conti$^{1}$, Patrick Dondl$^{2}$, Julia Orlik$^{3}$}\\[2mm]
{\em  $^{1}$  Institut f\"ur Angewandte Mathematik,
  Universit\"at Bonn,\\ 
  53115 Bonn, Germany}\\[1mm]
{\em    $^{2}$ Abteilung für Angewandte Mathematik,
 Albert-Ludwigs-Universität Freiburg, \\
 79104 Freiburg i. Br., Germany}\\[1mm]
 {\em $^{3}$ Fraunhofer-Institut für Techno- und Wirtschaftsmathematik ITWM\\ 
 67663 Kaiserslautern, Germany}
\end{center}
\begin{abstract}
\noindent We develop and analyze a variational model for multi-ply (i.e., multi-layered) paperboard. The model consists of a number of elastic sheets of a given thickness, which -- at the expense of an energy per unit area -- may delaminate. By providing an explicit construction for possible admissible deformations subject to boundary conditions that introduce a single bend, we discover a rich variety of energetic regimes. 
The regimes correspond to the experimentally observed: initial purely elastic response for small bending angle and the formation of a localized inelastic, delaminated hinge once the angle reaches a critical value. Our scaling upper bound then suggests the occurrence of several additional regimes as the angle increases.
The upper bounds for the energy are partially matched by scaling lower bounds.
\end{abstract}

	\section{Introduction and motivation}
Paperboard is an important engineering material, widely used for packaging, e.g., in the food industry. Due to its sustainability, paper-based materials have more recently gained in interest also for other applications \cite{Sim20}. Paperboard is essentially a comparatively thick  material made of processed wood pulp. Two types of paperboard can be distinguished, single-ply and multi-ply. The present article is concerned with multi-ply paperboard, which consists of multiple layers of paper bonded by an adhesive.

The deformation of paperboard is a complex process taking place on multiple scales. Deep drawing of paperboard to achieve a desired geometry -- similar to deep drawing of metals -- is an important technique that is the subject of current research \cite{LWO17}. For an overview of current modeling approaches, from the individual fiber level to the laminate structures in multi-ply paperboard, see \cite{Sim20}.

\begin{figure}
\begin{center}
\includegraphics[width=0.49\linewidth]{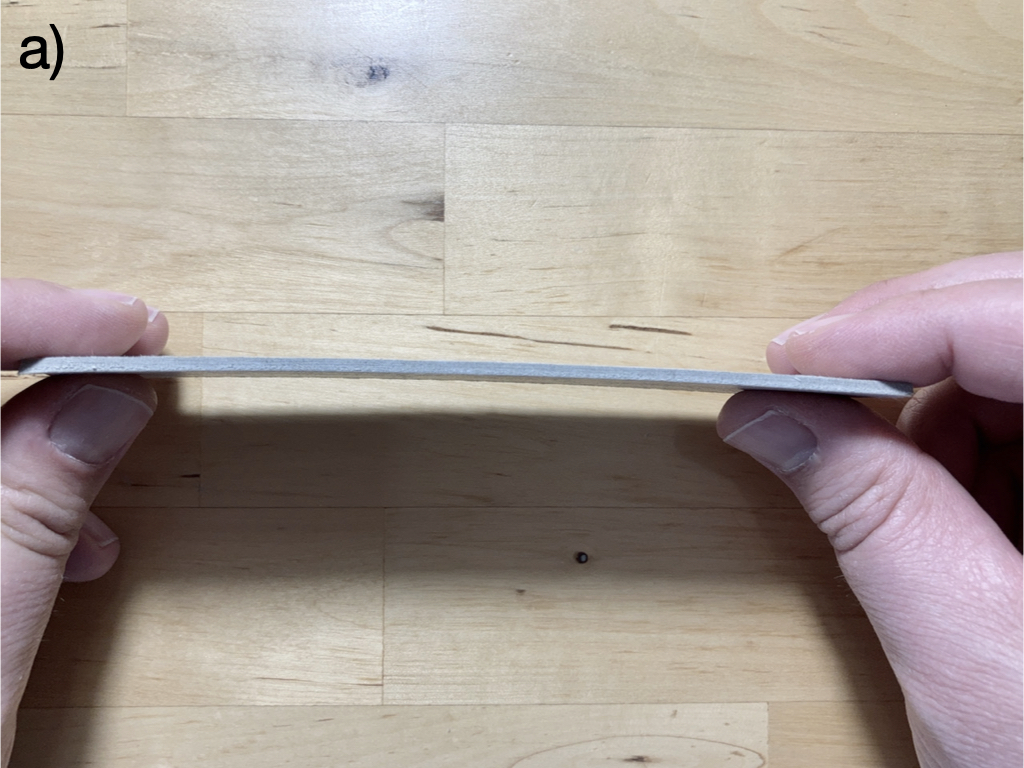}
\includegraphics[width=0.49\linewidth]{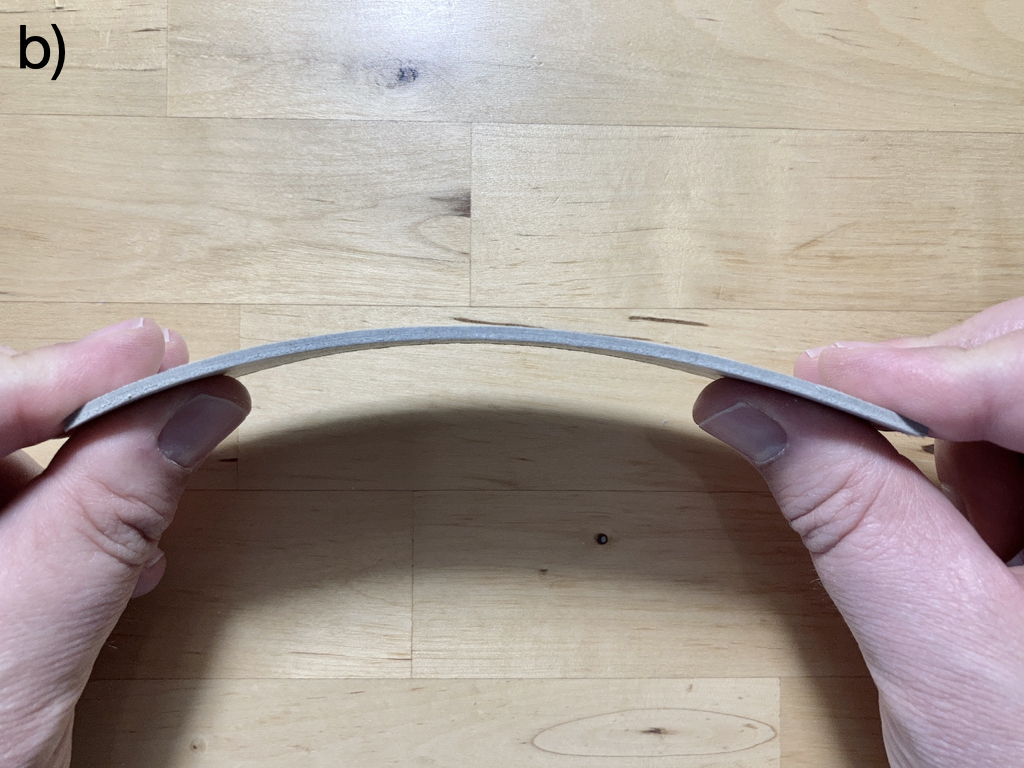}
\\[0.8mm]
\includegraphics[width=0.49\linewidth]{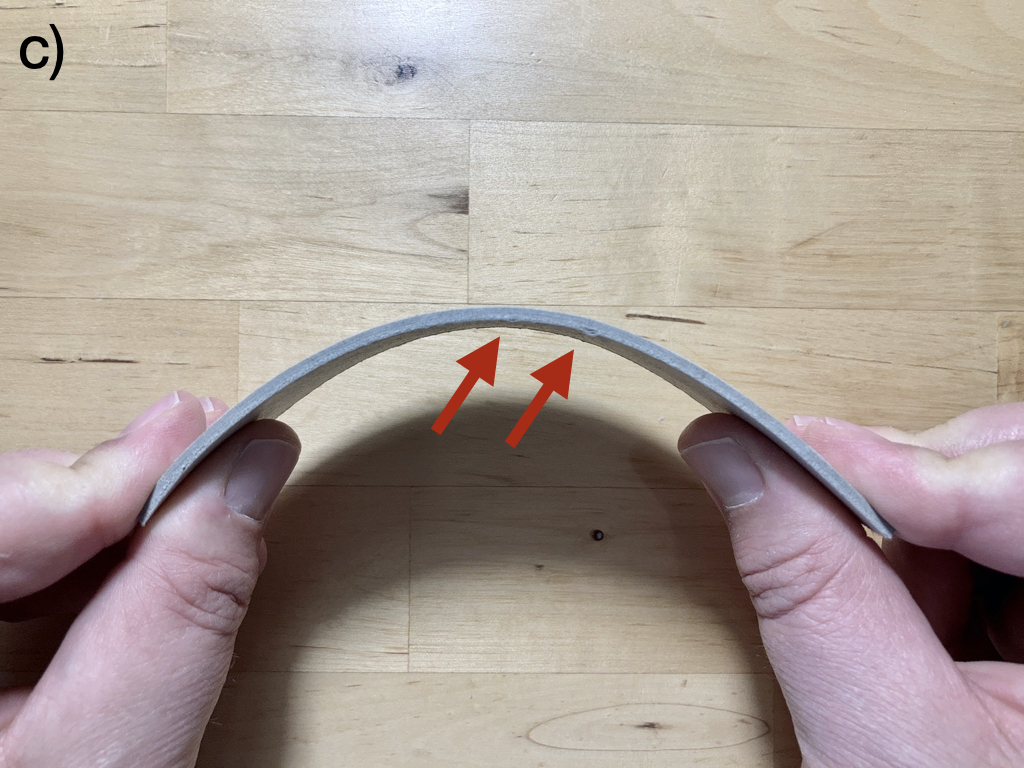} 
\includegraphics[width=0.49\linewidth]{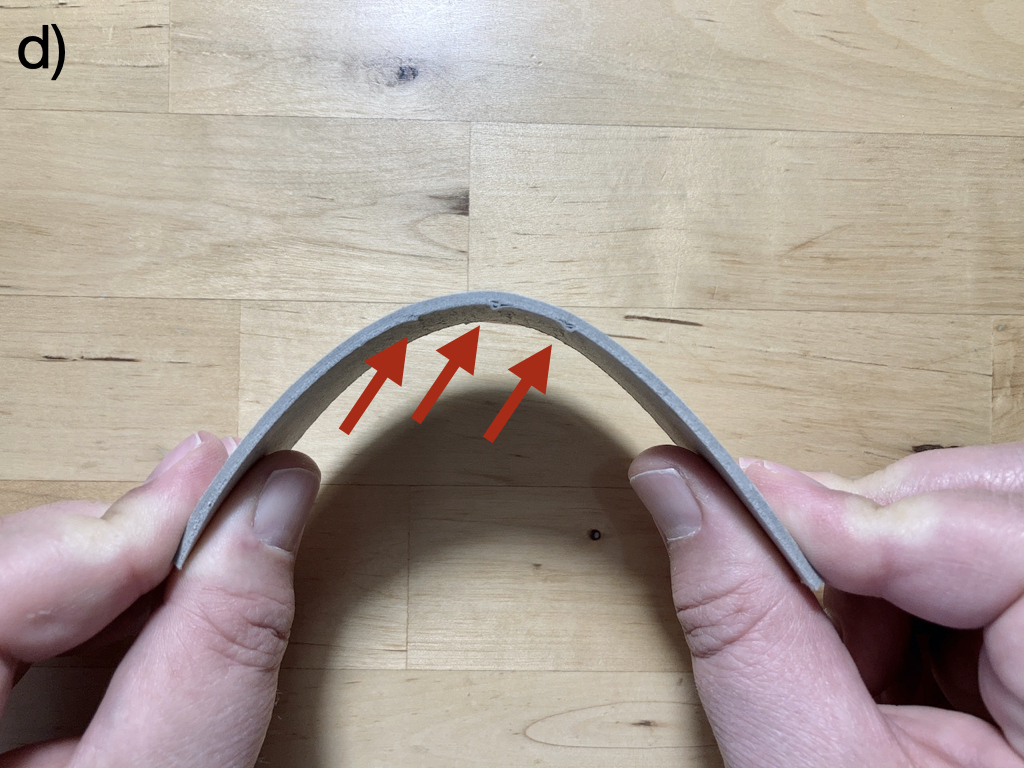}
\\[0.8mm]
\includegraphics[width=0.49\linewidth]{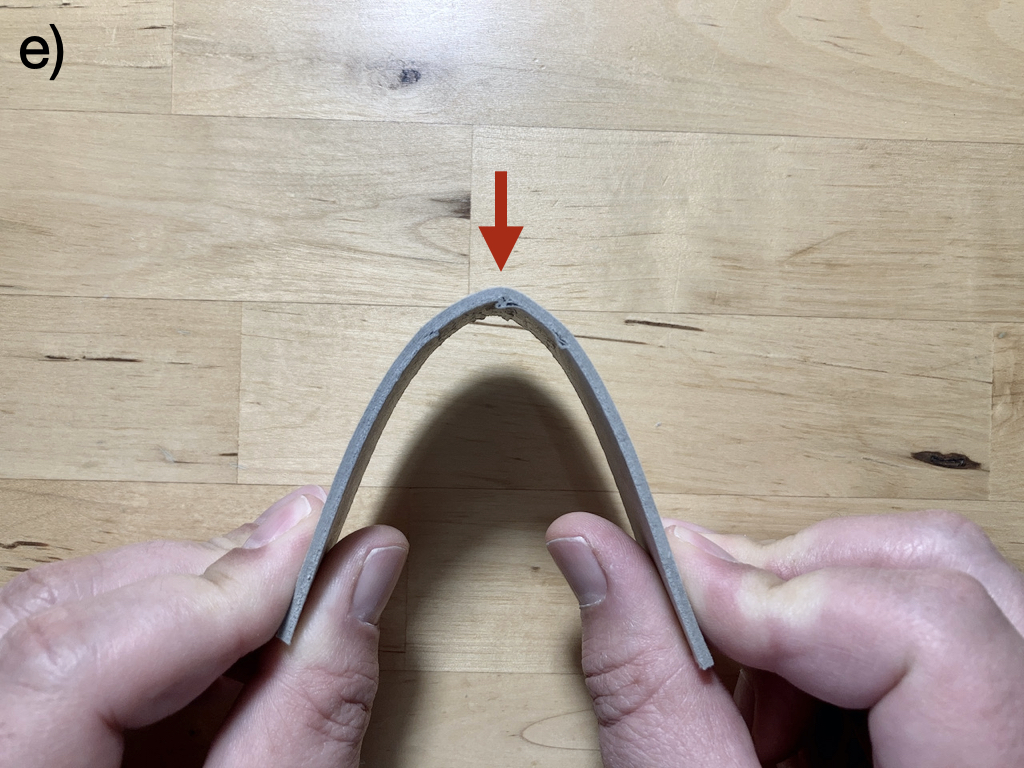}
\includegraphics[width=0.49\linewidth]{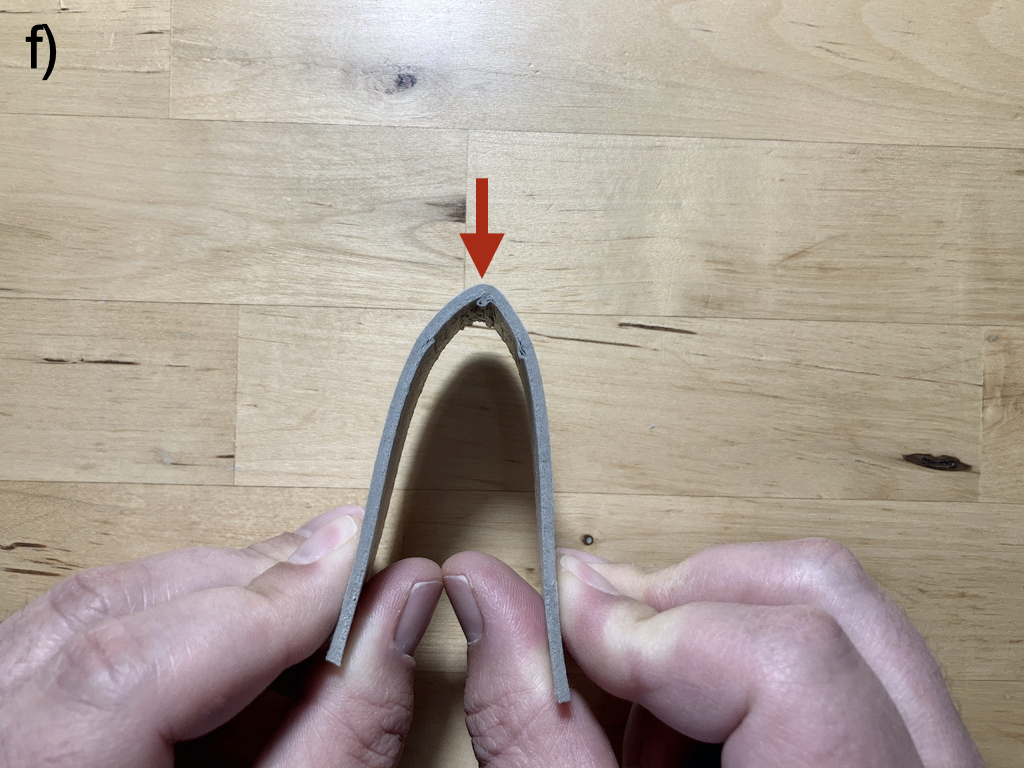}
\caption{Bending of multi-ply paperboard (BRAMANTE Buchbinderhartpappe 2mm) with increasing angle. a) unbent, b) purely elastic deformation, c,d) increasing delamination (red arrows), e),f) bending concentrates on delaminated hinge (red arrow). Images by D.\ Valainis.} \label{fig:bending_exp}
\end{center}
\end{figure}
\begin{figure}
\begin{center}
\includegraphics[width=0.48\linewidth]{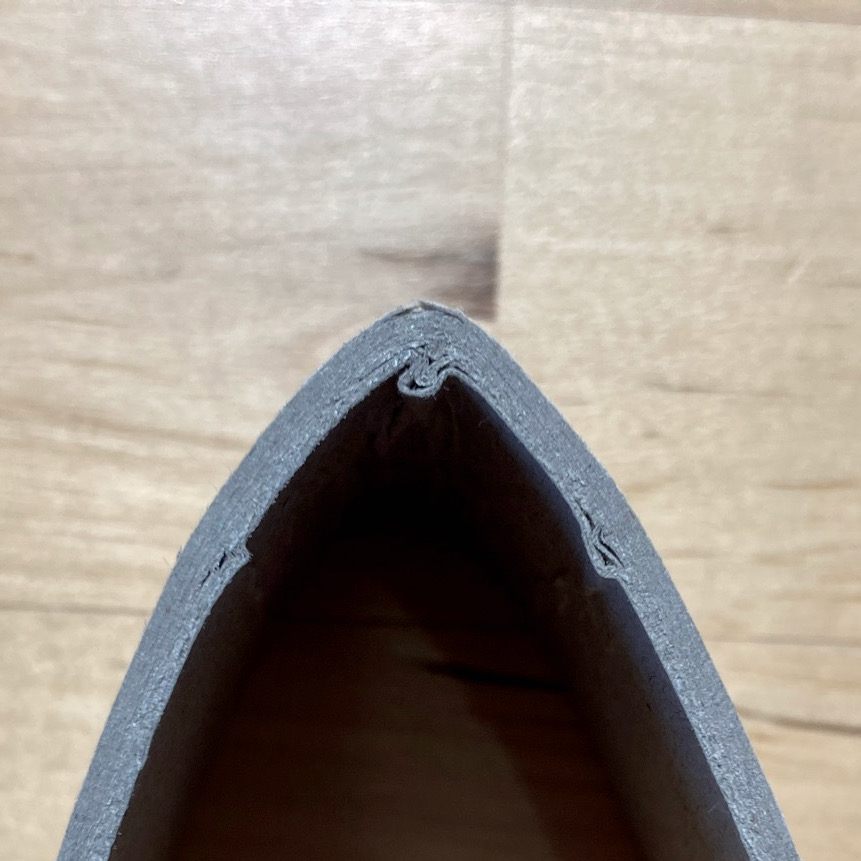}
\includegraphics[width=0.48\linewidth]{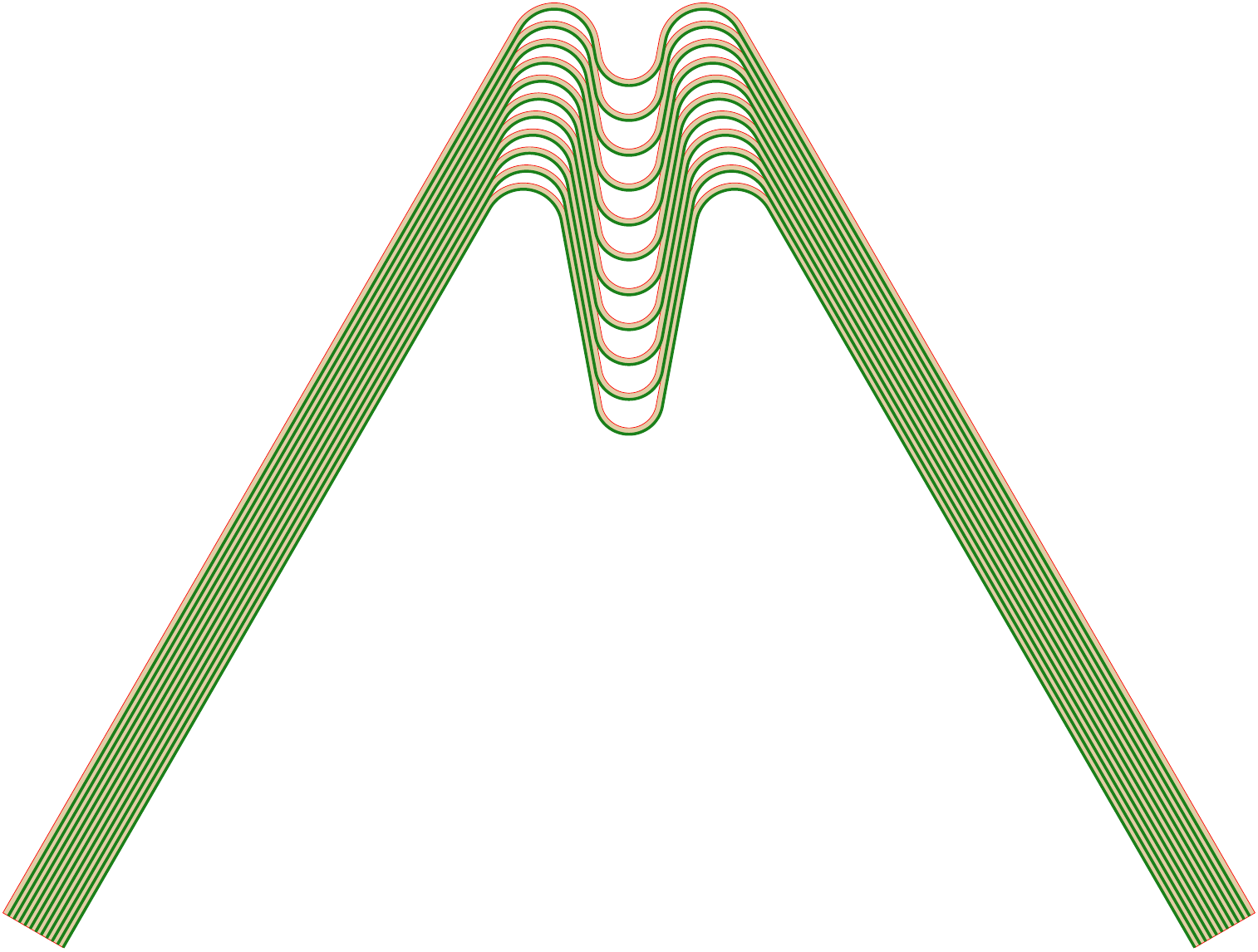}
\end{center}
\caption{Left: closeup of the hinge formed in the experiment shown in Figure \ref{fig:bending_exp}f). One can clearly see the concentration of the bending angle to the hinge as well as the delaminated layers of paperboard. Image by D.\ Valainis. Right: hinge construction with energetically optimal scaling.} \label{fig:hinge_exp}
\end{figure}

Of particular interest is the formation of individual hinges in multi-ply paperboard when it undergoes bending. In \cite{LWO17,Ost17, Hua11, BP09,BeexPeerlings2012}, the formation of such hinges has been closely examined. As one can see in Figure \ref{fig:bending_exp}, even in a simple experiment, such a hinge in multi-ply paperboard is characterized by a localized delamination of the individual sheets together with a localization of the deformation. A close up of the formed hinge is shown in Figure \ref{fig:hinge_exp} (left). One can clearly see that the delaminated layers on the inside of the bend bluckle, and thus release compressive stress. For experiments in a more controlled environment with an added crease to determine the exact location of the hinge, see, e.g., \cite[Figure 3]{BP09}. In particular in \cite{BP09,BeexPeerlings2012}, the experimental observations are complemented by a finite element model, where the delamination is treated using cohesive zones. A variational model, together with a rigorous mathematical treatment, resulting in a rich phase diagram of different energetic scaling regimes for multi-ply paperboard undergoing a simple bend, is provided in this article.

In our variational approach, the individual sheets are modeled using a nonlinearly elastic energy. The entire paperboard, but also of course the individual sheets of paper it consists of, are treated as slender elastic objects  \cite{Antman95,AudolyPomeau2010Buch}. More specifically, the rigorous methods used to derive plate energies from nonlinear bulk elasticity  \cite{LeDretRaoult95,FrieseckeJamesMueller2002, Lecumberry.2009} are used here to show lower scaling bounds for the deformed paperboard.

Delamination is treated as Griffith-type fracture, with a fixed energetic cost per delaminated unit area \cite{lawn1993fracture,bou-fra-mar}, but the fracture surface is restricted to the interfaces of the individual sheets of paper in the multi-ply paperboard. We refer to \cite{AmbrosioFuscoPallara} for an overview of the treatment of spaces of bounded variation which naturally occur in the mathematical treatment of fracture problems.

The competition between these two energies, a bulk elastic energy and a surface fracture energy, leads to the emergence of different scaling regimes, depending on parameters. In this sense, our article is inspired by the seminal work of Kohn and M\"uller \cite{KM94}. More closely related is the treatment of the blistering of thin films on a rigid substrate under compression \cite{BCDM00,JinSternberg2,BourneContiMueller2017}, where the formation of self-similar, branching channels was observed. The scaling regimes observed in the present work should be compared to the energy scaling for the folding of single sheets of paper \cite{Lobkovsky95,Venkataramani2004,ContiMaggi2008}.

Our main results can be summarized as follows. In our model, multi-ply paperboard exhibits different scaling regimes when subjected to bending boundary conditions, depending on the bending angle $\alpha$, the paperboard thickness $h$ and length $L$, number of sheets $N$ and Griffith energy coefficient $\gamma$.
For simplicity we focus here on the regime described by \eqref{eqecasessortedinallpha} below, which requires $h> \gamma N^3$, $h^5 N^5 > \gamma L^4$, and $h^5 < \gamma L^4 N^3$. Corresponding results for the other cases are discussed in Section~\ref{sec:upper}, see in particular Remark~\ref{eqmarkhsmall} and Remark~\ref{remarkubfll}.

\begin{itemize}
\item[-] First, for very small bending angle, no delamination occurs and a usual, energy minimizing single arch is formed. This purely elastic deformation, which corresponds to the bending of a thin plate, has an energy of order $\frac{\alpha^2h^3}{L}$.
\item[-] Then, sharply localized delamination occurs for all layers at a length scale $\frac{\alpha^{1/3}h^{4/3}}{\gamma^{1/3} N}$. The energy scales as $\alpha^{1/3}\gamma^{2/3}h^{4/3}$.
\item[-] For even larger bending angle, the delamination length increases to $\frac{\alpha h^{3/2}}{\gamma^{1/2}N^{3/2}}$ and the energy scales as $\frac{\alpha  \gamma^{1/2} h^{3/2}}{N^{1/2}}$.
\item[-] Finally, for larger bending angles, if the cost for delamination is small the entire paperboard may delaminate, which results in the standard bending energy for each of the $N$ sheets individually, yielding a total energy of $\frac{\alpha^2 h^3}{L N^2}$.
\end{itemize}
The occurrence of each individual regime depends on the magnitude of the model parameters. The second and third regime above correspond to the localized hinges observed in experiment. For most regimes, we provide rigorous lower bounds contingent on some assumptions on the delaminated sets. Explicit constructions realizing all energetic regimes by ensuring that individual delaminated sheets are provided. These constructions introduce an additional fold on the inside of the hinges, which allows the individual delaminated sheets to deform isometrically, as seen in the buckling visible in Figure \ref{fig:hinge_exp}, where the experiment is shown on the left, and a construction with energetically optimal scaling is shown on the right.

The remainder of this article is organized as follows. In Section \ref{sec:model} we introduce our mathematical model. Scaling upper and lower bounds for the energy are proved in Sections \ref{sec:upper} and \ref{sec:lower}, respectively. We close with a brief discussion in Section \ref{sec:discussion}.

\section{Problem setting and model} \label{sec:model}

\subsection{The three-dimensional model}

We consider a paperboard sample of thickness $h$, consistings of $N$ layers, so that at most $N-1$ delamination surfaces are possible. We let $2L$ be the length of the sample in the direction of bending, and consider a section of unit length in the third direction, which will be irrelevant for our results. We thus obtain a reference configuration $\Omega_h:=(-L,L)\times(0,h)\times(0,1)$. 

The set of admissible deformations consists of maps 
$U:\Omega_h\to\R^3$ which jump only on (a subset of) $N-1$ prescribed planes which correspond to the
delamination surfaces between layers, in the sense that 
the set $J_U$ of
jump points of $U$ obeys
(up to null sets)
\begin{equation}\label{eqjumpplanes}
J_U\subset \Omega_h\cap \{x: x_2\in \frac{h}{N}\Z\}=(-L,L)\times\{\frac hN, \frac{2h}N, \dots, \frac{(N-1)h}N\}\times (0,1).
\end{equation}
Outside its jump set the function is assumed to have a weak gradient which is square integrable. 
Mathematically, this means that $U$ belongs to $ SBV^2_N(\Omega_h;\R^3)$, which we define as the space of functions in $SBV(\Omega_h;\R^3)$
such that \eqref{eqjumpplanes} holds and $\nabla U \in L^2(\Omega_h;\R^{3\times 3})$. We recall that $SBV(\Omega_h;\R^3)$ is the set of special functions of bounded variation, i.e., the set of integrable functions $U:\Omega_h\to\R^3$ such that the distributional gradient $DU$ is a measure  of the form
$DU=\nabla U\calL^3+[U]\otimes \nu \calH^2\LL J_U$, with $\nabla U\in L^1(\Omega_h;\R^{3\times 3})$, $J_U$ the $\calH^2$-rectifiable jump set of $U$, $[U]:J_U\to\R^3$ its jump, and $\nu:J_U\to S^2$ the normal to $J_U$. We refer to
\cite{AmbrosioFuscoPallara} for details of this definition and the relevant properties of the space.

In order to introduce a boundary condition for the hinge, we prescribe the deformation gradient on both ends of the domain, i.e.,   
\begin{equation}\label{eqboundarysinglefold}
DU(x)= \hat R_\alpha \text{ for } x_1<-\frac{L}{2} \,,\hskip1cm
DU(x)= \hat R_{-\alpha} \text{ for } x_1>\frac{L}{2} \,,
\end{equation}
where $\hat R_\alpha\in\SO(3)$ is a rotation of angle $\alpha$ around $x_3$,
\begin{equation}
    \hat R_\alpha = \begin{pmatrix} \cos\alpha & -\sin\alpha & 0\\
    \sin\alpha & \cos\alpha & 0\\
    0 & 0 & 1
    \end{pmatrix}.
\end{equation}

The energy of a deformation $U\in SBV^2_N(\Omega_h;\R^3)$ consists of the sum of an elastic energy, depending on the absolutely continuous part of the deformation gradient $\nabla U$,  and a delamination energy, which is proportional to the total area of the delaminated set $J_U$. Specifically,
\begin{equation}
\energy^\mathrm{3D}_h[U]:=\int_{\Omega_h} \Wddd(\nabla U) dx + \gamma {\cal H}^2(J_U),\;\;\text{ for } U\in SBV_N^2(\Omega_h;\R^3).
	\end{equation} 
Here $\Wddd:\R^{3\times 3}\to[0,\infty)$ is an elastic energy density which obeys for some $c>0$
\begin{equation}\label{eqgrowthwdd}
    \frac1c \dist^2(\xi, \SO(3)) \le \Wddd (\xi) \le c \dist^2(\xi,\SO(3)) \text{ for all } \xi\in \R^{3\times 3}.
\end{equation}
The parameter $\gamma>0$ is the delamination energy per unit area, which in the present setting has the dimensions of a length, 
and	${\cal H}^2$ denotes the Hausdorff measure.
In particular, ${\cal H}^2(J_U)$ is the area of the delaminated set.

\subsection{Reduction to two dimensions}

In this paper we assume that no structure arises in the $x_3$-direction, in the sense that the deformation $U$ takes the form
\begin{equation}
    U(x_1,x_2,x_3)=u(x_1,x_2)+x_3e_3
\end{equation}
for some $u\in SBV^2_N(\omega_h;\R^2)$, with $\omega_h:=(-L,L)\times (0,h)$.
The latter set is defined as the set of $SBV$ functions such that $\nabla u\in L^2(\omega_h;\R^{2\times 2})$ and
\begin{equation}\label{eqjumpplanes2d}
J_u\subset \{x\in \omega_h: x_2\in \frac{h}{N}\Z\}=(-L,L)\times\{\frac{h}{N}, \frac{2h}{N}, \dots, \frac{(N-1)h}{N}\},
\end{equation}
and one easily sees that $U\in SBV^2_N(\Omega_h;\R^3)$ if and only if $u\in SBV^2_N(\omega_h;\R^2)$.

The energy then reduces to
\begin{equation}
\energy_h[u]:=\energy^\mathrm{3D}_h[U]=\int_{(-L,L)\times (0,h)} \Wdd(\nabla u) dx+ \gamma \calH^1(J_u), \text{ for } u\in SBV^2_N(\omega_h;\R^2),
\end{equation}
where $\Wdd$ is defined by $\Wdd(\xi):=\Wddd(\xi + e_3\otimes e_3)$ for  $\xi\in \R^{2\times 2}$, 
and ${\cal H}^1$ denots the one-dimensional Hausdorff measure, which measures length. Here and below we identify 
$\xi\in \R^{2\times 2}$ with the matrix $\hat\xi\in\R^{3\times 3}$ characterized by $\hat\xi_{ij}=\xi_{ij}$ for $i,j=1,2$, zero otherwise, and correspondingly for vectors. From \eqref{eqgrowthwdd} one easily obtains that the two-dimensional reduced energy obeys
\begin{equation}\label{eqW2dbounds}
    \frac1c \dist^2(\xi, \SO(2)) \le \Wdd (\xi) \le c \dist^2(\xi,\SO(2)) \text{ for all } \xi\in \R^{2\times 2}.
\end{equation}

The boundary condition \eqref{eqboundarysinglefold} in  turn is equivalent to
\begin{equation}\label{eqboundarysinglefold2d}
Du(x)= R_\alpha \text{ for } x_1<-\frac{L}{2} \,,\hskip1cm
Du(x)= R_{-\alpha} \text{ for } x_1>\frac{L}{2} \,,
\end{equation}
where we denote by $R_\alpha\in\SO(2)$ the two-dimensional rotation of angle $\alpha$,
\begin{equation}\label{eqdefRalpha2d}
    R_\alpha:= \begin{pmatrix} \cos\alpha & -\sin\alpha \\
    \sin\alpha & \cos\alpha 
    \end{pmatrix}.
\end{equation}

\section{Scaling upper bounds for the energy} \label{sec:upper}

	 The objective is to find {\bf critical bending angle} and {\bf energy bounds} for interplay of bending and delamination in terms of  
	$h$, $L$, $\gamma$.

\subsection{First construction: Plate bending without delamination.} \label{sec:no-delam}

We first consider the case that no delamination occurs. Then it is natural to expect that the deformation has constant curvature. 
We use here the classical construction for plate theories, with some simplifications which do not alter the scaling of the energy (for example, 
in \eqref{eqdefufromvplate} we simply use $x_2$ as a factor, and not 
 $x_2-\frac h2$). We refer to \cite{FrieseckeJamesMueller2002} for a more general mathematical treatment.
\begin{lemma}\label{lemmacontinuousbending}
For all $h,L>0$ and $\alpha\in[0,\frac\pi2]$ there is a map $u\in W^{1,\infty}(\omega_h;\R^2)$ which obeys the boundary condition  \eqref{eqboundarysinglefold2d} and such that
\begin{equation}
      \energy_h[u]\le c\frac{\alpha^2h^3}{L}.
\end{equation}
The map $u$ is injective, and
$W^{1,\infty}(\omega_h;\R^2)\subseteq SBV^2_N(\omega_h;\R^2)$ for all $N$.
\end{lemma}
\begin{proof}
We first fix an arc of circle for the central part,
\begin{equation}
    f(x_1):=
    \frac {L}{2\alpha} 
    \begin{pmatrix} \sin (2\alpha x_1/L)\\ \cos (2\alpha x_1/L) 
    \end{pmatrix},
\end{equation}
which obeys $|f'(x_1)|=1$ for all $x_1$, 
and then extend it piecewise affine in the two boundary regions, setting
\begin{equation}
    v(x_1):=
    \begin{cases}
    f(-\frac L2) + (x_1+\frac L2) f'(-\frac L2), & \text{ if }-L<x_1<-\frac L2,\\
    f(x_1), & \text{ if } -\frac L2\le x_1\le \frac L2,\\
    f(\frac L2) + (x_1-\frac L2) f'(\frac L2), & \text{ if }\frac L2<x_1<L.
    \end{cases}
\end{equation}
One easily verifies that  $v\in W^{2,\infty}((-L,L);\R^2)$, with $|v'|=1$ and  $|v''|\le 2\alpha/L$ almost everywhere.
We  then define $u:\omega_h\to\R^2$ by
\begin{equation}\label{eqdefufromvplate}
    u(x_1,x_2):=v(x_1) + x_2 (v')^\perp(x_1),
\end{equation}
where $(a_1,a_2)^\perp:=(-a_2,a_1)$ denotes counterclockwise rotation by 90 degrees.
One easily verifies that $u\in W^{1,\infty}(\omega_h;\R^2)$, $Du(x_1,x_2)=R_{\mp\alpha}$ for $\pm x_1\in (\frac L2,L)$, and, for $x_1\in (-\frac L2, \frac L2)$
\begin{equation}
    Du(x)=v'(x_1)\otimes e_1 + (v')^\perp(x_1) \otimes e_2 + x_2 (v'')^\perp(x_1)  \otimes e_1
\end{equation}
	so that $\dist(Du(x),\SO(2))\le |x_2|\, |v''(x_1)|$. Recalling the upper bound in 
	\eqref{eqW2dbounds} we obtain the desired bound
	\begin{equation}\label{eqestenergyplate}
	\begin{split}
\energy_h[u]\le& \int_{-L}^L \int_{0}^{h} c | x_2 v''(x_1) |^2 dx_2 dx_1\\
	      =&  L \int_{0}^{h} c\left( \frac{2\alpha x_2}{L}\right)^2 dx_2
	      \le c\frac{\alpha^2 h^3}{L}.
	      \end{split}
	\end{equation}
	Injectivity can be easily verified from the definition of $u$.
	\end{proof}
	
\subsection{Continuous, piecewise affine construction}
\label{secpwaffine}
\newcommand\ucaff{u^\mathrm{cpa}}
We start by constructing a continuous piecewise affine map $\ucaff$ that illustrates the basic structure of the deformation. 
This  construction is not admissible for the functional considered here, but represents a deformation for the limiting case $N\to\infty$.
Before starting we introduce the notation
\begin{equation}
    R_\varphi := \begin{pmatrix}\cos \varphi & -\sin \varphi  \\  \sin \varphi & \cos \varphi  \\ \end{pmatrix} \text{ for } \varphi\in\R.
\end{equation}

We fix $\alpha\in(0,\frac\pi2)$, $h,L>0$.
We seek a map $\ucaff$  that obeys the boundary conditions \eqref{eqboundarysinglefold2d}, and that is symmetric with respect to the $x_2$-axis, in the sense that
\begin{equation}\label{eqsymm}
    \ucaff(-x_1,x_2)=\begin{pmatrix} -\ucaff_1(x_1,x_2)\\ \ucaff_2(x_1,x_2) \end{pmatrix}.
\end{equation}
Therefore we can focus on $\{x_1\ge0\}$, provided that we impose the condition $ \ucaff_1(0,x_2)=0$.
We fix a parameter $\zeta\in(0,L/(2h)]$, and assume that 
\begin{equation}
    D\ucaff=R_{-\alpha} \text{ for } \zeta (h-x_2)<x_1<L, \hskip5mm 0<x_2<h,
\end{equation}
where $R_{-\alpha}$ was defined in \eqref{eqdefRalpha2d} (see Figure~\ref{fig:pw_aff} for a sketch of the geometry).
This ensures that the boundary condition \eqref{eqboundarysinglefold2d} is fulfilled.
In the region $0<x_1<\zeta (h-x_2)$ the material is allowed to shear and to open across the possible delamination lines, and of course to rotate by some angle $\beta$ to be determined. However, the non-interprenetration condition prevents volumetric compression. In other words we assume that the deformation gradient takes the form
\begin{equation}\label{eqFRbetaa1a2}
F^*:=R_\beta \begin{pmatrix}  1 & a_1\\ 0 & a_2 \end{pmatrix} 
\end{equation}
for some $a_1\in\R$ (representing shear) and $a_2\in[1,\infty)$ (representing opening across the delamination lines). 
The symmetry condition \eqref{eqsymm} requires that 
$\ucaff_1(0,x_2)=0$ for all $x_2\in(0,h)$, which
-- given $\ucaff_1(0,h)=0$ -- 
is equivalent to
$F^*_{12}=0$ and therefore to
$a:=(a_1,a_2)=|a|(\sin\beta,\cos\beta)$. Denoting $d:=|a|$ we obtain
\begin{equation}
F^*=R_\beta \begin{pmatrix}  1 & d \sin\beta \\ 0 & d\cos\beta \end{pmatrix} 
= \begin{pmatrix}  \cos\beta  & 0 \\ \sin\beta & d \end{pmatrix} 
\end{equation}
with the condition $d\cos\beta\ge1$.
The construction is concluded if we impose continuity across the line $\{x_1=\zeta(h-x_2)\}$, which is equivalent to
\begin{equation}
    0=(R_{-\alpha}-F^*) \begin{pmatrix} \zeta \\ -1 \end{pmatrix} 
    =\begin{pmatrix}
    \zeta\cos\alpha-\sin\alpha -\zeta\cos\beta\\
    -\zeta\sin\alpha-\cos\alpha-\zeta\sin\beta+d
    \end{pmatrix}.
\end{equation}
In turn, this can be rewritten as the two conditions
\begin{equation}\label{eqpcadefgamma}
    \zeta=\frac{\sin\alpha}{\cos\alpha-\cos\beta}
\end{equation}
and
\begin{equation}\label{eqpcadefd}
    d=\cos\alpha + \zeta(\sin\beta+\sin\alpha) = \frac{1-\cos\alpha\cos\beta+\sin\alpha\sin\beta}{\cos\alpha-\cos\beta}
\end{equation}
which define $\zeta$ and $d$ in terms of the angle $\beta$.
It remains to verify the  conditions $a_2=d\cos\beta\ge1$ and $0\le \zeta\le L/(2h)$, which limit the admissible choices of the angle $\beta$.  We observe that this condition is needed to ensure injectivity of the final construction.
We assume that $\alpha,\beta\in(0,\frac\pi2)$; by $\zeta>0$ we necessarily have $\beta>\alpha$. 

The condition $d\cos\beta\ge1$ can be rewritten as
\begin{equation}\label{eqdeffalhabeta}
   f_\alpha(\beta):=  \frac{1-\cos\alpha\cos\beta+\sin\alpha\sin\beta}{\cos\alpha-\cos\beta}\cos\beta\ge1.
\end{equation}
A short computation shows that 
\begin{equation}\label{fabmp}
   f_\alpha(\beta)=  \frac{\sin \frac{\alpha+\beta}{2}}{\sin \frac{\beta-\alpha}{2}}\cos\beta 
\end{equation}
and
\begin{equation}
   f_\alpha'(\beta)=  \frac{(\cos\beta-\cos\alpha)\sin\beta - \cos\beta\sin\alpha}{2\sin^2 \frac{\alpha-\beta}{2}}<0
\end{equation}
so that for any $\alpha\in(0,\frac\pi2)$ the function $f_\alpha:(\alpha,\frac12\pi]\to\R$ is strictly decreasing
with $f_\alpha(\frac\pi2)=0$ and $f_\alpha(\alpha^+)=\infty$. Therefore we can define 
 $\betaeq(\alpha)\in (\alpha, \frac\pi2)$ as the unique solution 
 to $f_\alpha(\beta)=1$. By 
 the implicit function theorem, $\betaeq\in C^1((0,\frac\pi2))$, and from
 \begin{equation}
  \frac{\partial f_\alpha(\beta)}{\partial \alpha} = 
\frac{  \sin\beta\cos\beta}{1-\cos(\alpha-\beta)}>0
 \end{equation}
we obtain  $\betaeq'>0$. The condition $d\cos\beta\ge1$ is equivalent to $\beta\in (\alpha,\betaeq(\alpha)]$.
We summarize and extend these results in the following statement.

\begin{lemma}\label{lemmabetaequcaff}
(i) There is a continuous, increasing function $\betaeq:(0,\frac\pi2)\to(0,\frac\pi2)$ such that $\alpha<\betaeq(\alpha)$ for all $\alpha$ and
\begin{equation}\label{eqdeffalhabeta2}
    \frac{1-\cos\alpha\cos\betaeq(\alpha)+\sin\alpha\sin\betaeq(\alpha)}{\cos\alpha-\cos\betaeq(\alpha)}\cos\betaeq(\alpha)=1.
\end{equation}
It obeys
\begin{equation}\label{eqbetaeqsmallalpha}
    \betaeq(\alpha)=4^{1/3}\alpha^{1/3}+o(\alpha^{1/3}) \hskip5mm\text{ as } \alpha\to0.
\end{equation}

(ii) For any $\alpha,\beta\in(0,\frac\pi2)$, $h,L>0$, if 
\begin{equation}
    \alpha<\beta\le\betaeq(\alpha) \text{ and } \zeta:=\frac{\sin\alpha}{\cos\alpha-\cos\beta} \le \frac{L}{2h}
\end{equation}
then the map $\ucaff:\overline\omega_h\to\R^2$ defined by
\begin{equation}\label{eqdefucaff}
    \ucaff(x):=\begin{cases}
    \begin{pmatrix} x_1\cos\beta \\ |x_1|\sin\beta +dx_2 \end{pmatrix},  & \text{ if } |x_1|<\zeta (h-x_2),\\
\begin{pmatrix} x_1\cos\alpha + (x_2-h)\sin\alpha \sgn x_1\\ 
-|x_1|\sin\alpha + (x_2-h)\cos\alpha +dh  \end{pmatrix}, & \text{ if }  \zeta (h-x_2)\le |x_1|\le L,
    \end{cases}
\end{equation}
is continuous, injective, obeys the boundary conditions and $|\partial_1 \ucaff|=1$ almost everywhere.
\end{lemma}

	\begin{proof}
	(i): The function $\betaeq$ is defined by $f_\alpha(\betaeq(\alpha))=1$, with $f$ as in \eqref{eqdeffalhabeta}. 
	In order to verify \eqref{eqbetaeqsmallalpha}, we write 
	\eqref{eqdeffalhabeta2} as
\begin{equation}
    (1-\cos\alpha\cos\beta+\sin\alpha\sin\beta)\cos\beta={\cos\alpha-\cos\beta},
\end{equation}
	expand both sides to second order in $\alpha$ and rearrange terms, to obtain
	\begin{equation}
	    \alpha \sin\beta\cos\beta 
	    +\frac12\alpha^2(\cos^2\beta+1)
	    = (1-\cos\beta)^2 + O(\alpha^3).
	\end{equation}
	In particular, $\lim_{\alpha\to0}\betaeq(\alpha)=0$.
	Expanding to leading order in $\beta$ 
	leads to $\alpha \beta +\alpha^2= (\beta^2/2)^2 + O(\alpha^3)+O(\alpha\beta^3)+O(\alpha^2\beta^2)+O(\beta^6)$, which implies 
	 \eqref{eqbetaeqsmallalpha}.
	
	(ii): We define $\zeta$ and $d$ by \eqref{eqpcadefgamma} and \eqref{eqpcadefd}, respectively. 
	One verifies that $\ucaff$ as defined in \eqref{eqdefucaff} obeys the symmetry condition \eqref{eqsymm}, $\ucaff_1(0,x_2)=0$, and $D\ucaff=F^*$ for $0<x_1<\zeta(h-x_2)$,
	$D\ucaff=R_{-\alpha}$ for $x_1\ge\zeta(h-x_2)$, and that the function is continuous on the point $(0,h)$. As is apparent from the construction above (see in particular \eqref{eqFRbetaa1a2}) 
this map is injective and transforms horizontal lines in a length-preserving way, in the sense that $|\partial_1 \ucaff|=1$.
 The rest follows from the computations above.
	\end{proof}
	
	\subsection{Second construction: Multilayered folding with partial delamination. }

\begin{figure}[t]
    \centering
    \includegraphics[width=0.8\linewidth]{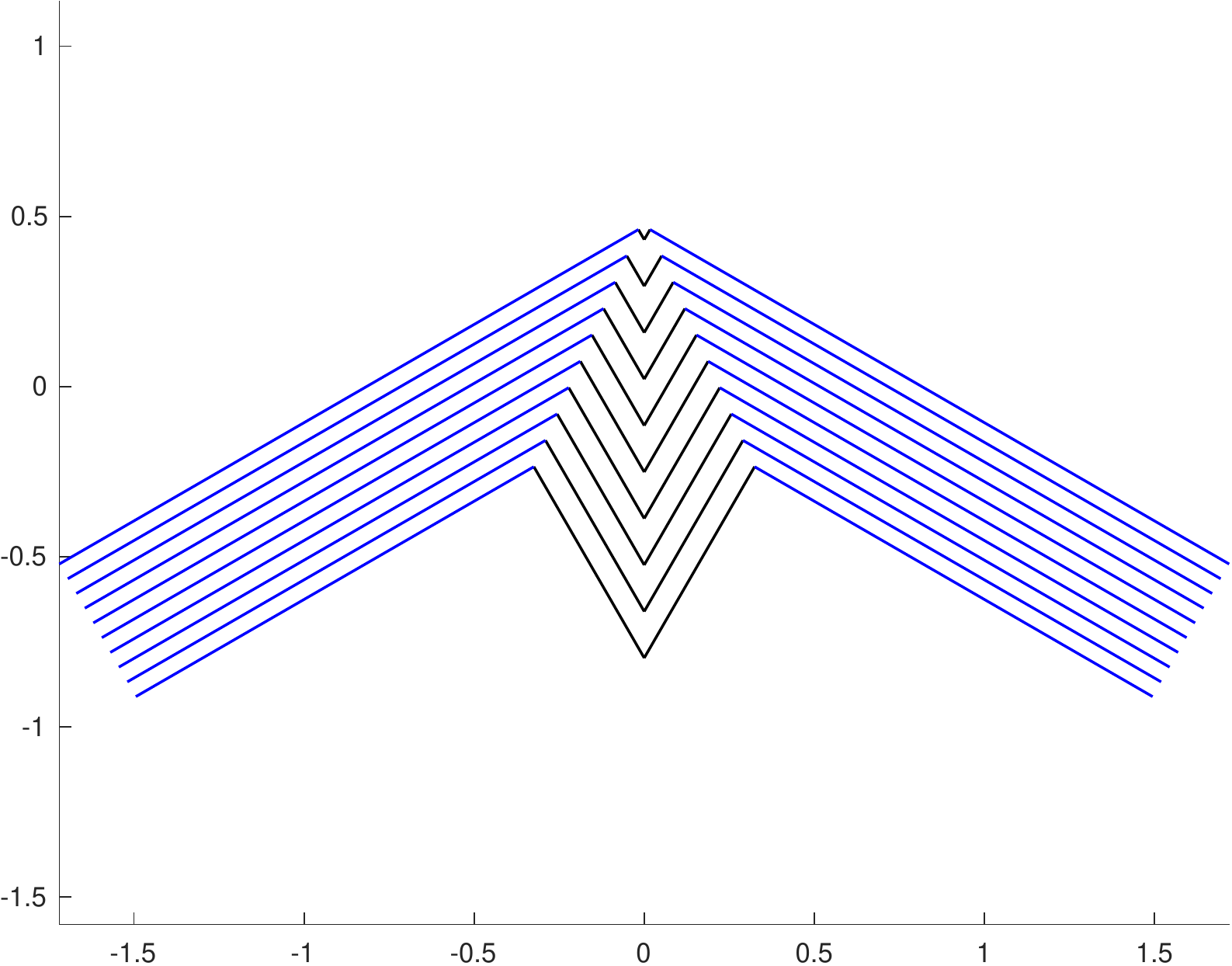}
    \caption{Piecewise affine construction. The angle formed by the segments with the $x$ axis is $\alpha$ on the left, $-\beta$ in the first (downward) part of the central fold, then $\beta$, and finally $-\alpha$ on the right.}
    \label{fig:pw_aff}
\end{figure}

	The construction arises as a discretization and regularization of the continuous piecewise affine construction of Section \ref{secpwaffine}. The final result of the construction, for two choices of $\beta$, is illustrated in Figure \ref{fig:2d_construction}.
	
	We assume that in the inner region the material is partially delaminated, and use this to replace the non-isometric gradient $F^*$ by a deformation which is isometric away from the discontinuity set. We then replace the sharp corners by regularized corners, in which each layer smoothly bends from angle $\alpha$ the angle $-\beta$, then to $\beta$, and finally to $-\alpha$.

	We consider a construction with $n\le N$  layers of paperboard which have been separated across $n-1$ delamination surfaces. We assume $n\ge 1$; the case $n=1$ without delamination has already been treated in Section~\ref{sec:no-delam}.
	To this end, we fix a subset  $\{\delamy_1 < \delamy_2< \dots <\delamy_{n-1}\}\subseteq (0,h)\cap \frac hN\Z$ and assume that delamination occurs only on the surfaces $\{x_2=\delamy_1, \dots \delamy_{n-1}\}$, in the sense that $J_u\subseteq (-L,L)\times \{\delamy_1, \dots, \delamy_{n-1}\}$. For notational convenience we denote $\delamy_0:=0$ and $\delamy_{n}:=h$.
	We label by $h_j:=\delamy_{j+1}-\delamy_j$, $0\le j< n$, the thickness of the $j$-th layer.
	The map we construct will be in $C^1( [-L,L]\times [\delamy_j, \delamy_j+h_j))$ for each $j< n$.

    \paragraph{Step 1.} The first step is a piecewise affine construction. 
    We first construct the map on the set $\{x_2=\delamy_j\}$ using the continuous piecewise affine construction as background. Recalling \eqref{eqdefucaff} we define maps $\hat f_0,\dots, \hat f_{n-1}:[-L,L]\to \R^2$ by
    \begin{equation}\label{eqdeffj}
        \hat f_j(x_1):=\ucaff(x_1,\delamy_j)=
        \begin{cases}
    \begin{pmatrix} x_1\cos\beta \\ |x_1|\sin\beta +db_j \end{pmatrix},  & \text{ if } |x_1|<l_j,\\
\begin{pmatrix} x_1\cos\alpha + (\delamy_j-h)\sin\alpha \sgn x_1\\ 
-|x_1|\sin\alpha + (\delamy_j-h)\cos\alpha +dh  \end{pmatrix}, & \text{ if }  l_j\le |x_1|\le L,
    \end{cases}
    \end{equation}
    where $l_j:=\zeta(h-\delamy_j)$
    and $\zeta$ is defined in \eqref{eqpcadefgamma}.
    These maps are illustrated in Figure  \ref{fig:pw_aff}. The inner part of the construction, for $|x_1|\le l_j$, will be referred to as the ``down slope''.
    
    \paragraph{Step 2.} In a next step we round the corners. This could be done by mollification, but it is important to (i) check the length, and keep the deformation isometric in the longitudinal direction, and (ii) be able to verify global injectivity of the two-dimensional deformation. Therefore we take a more explicit path and insert a circular arc at each of the points of discontinuity of $\hat{f}'_j$, namely, at $x_1=\pm l_j$ and at $x_1=0$.
    
\begin{figure}[t]
\begin{center}    
\includegraphics[width=0.8\linewidth]{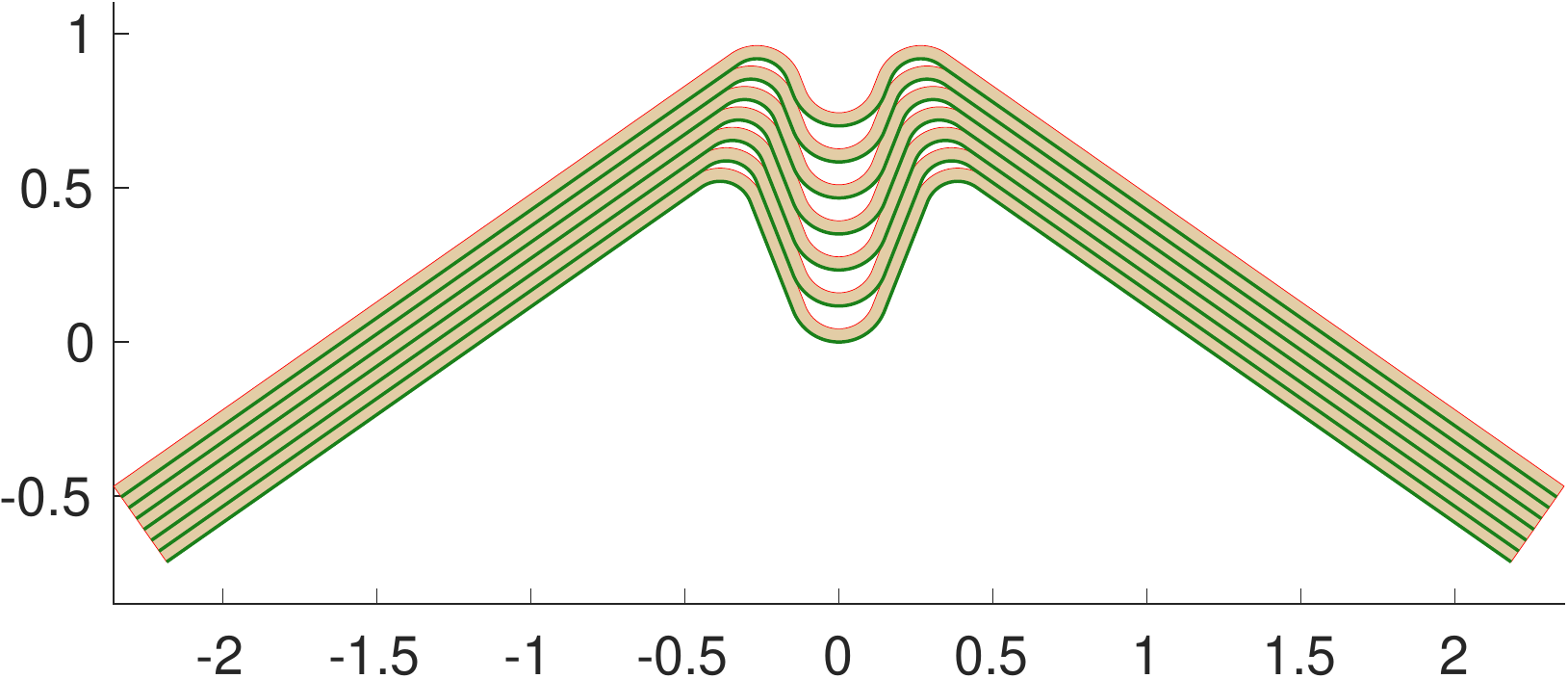}\\
     \includegraphics[width=0.8\linewidth]{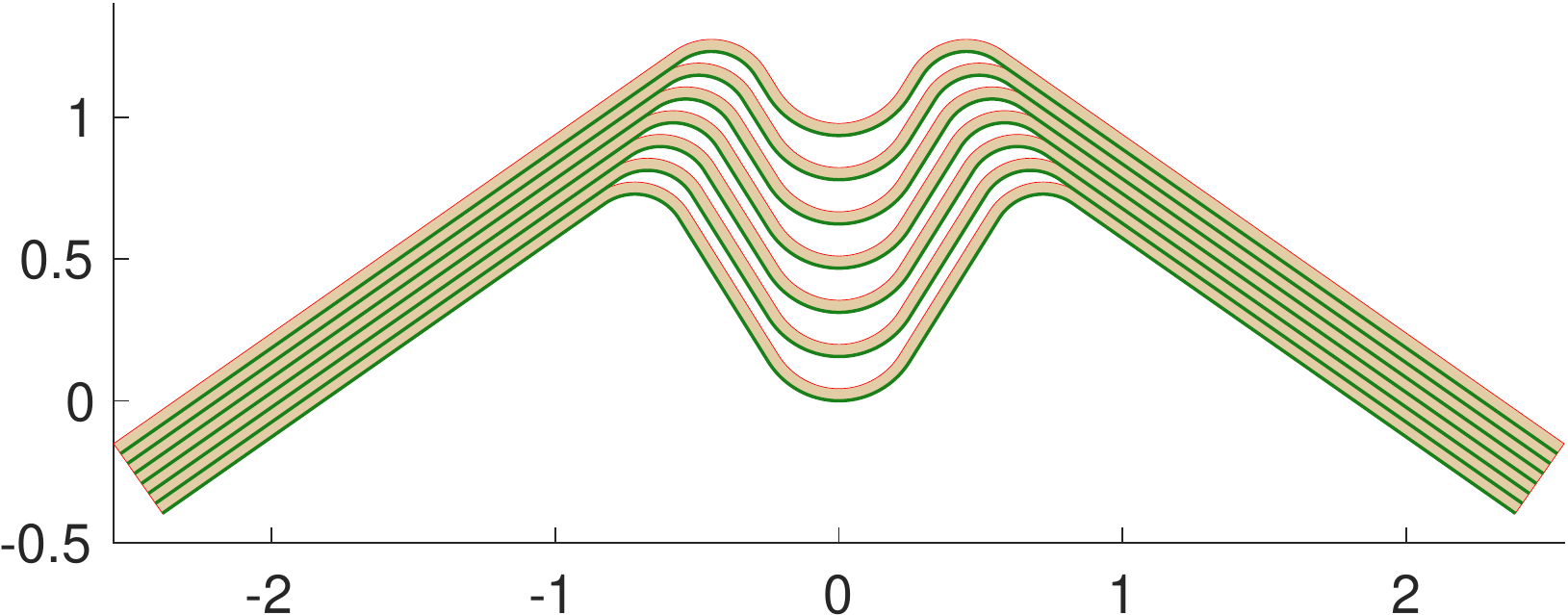}
     \end{center}
    \caption{Construction of the multilayered folding construction. Top: $\beta=\betaeq$, bottom: $\beta<\betaeq$.
    This is a smoother version of the backbone structure from Figure~\ref{fig:pw_aff}. Also in this case the angle formed by the straight segments with the $x$ axis is $\alpha$, $-\beta$, $\beta$ and $\alpha$ (from left to right).}
    \label{fig:2d_construction}
\end{figure}
    
    The length for each of the four arcs is chosen equal as $\larc$. The centers as well as radii are chosen such that the tangents of the arc match the respective left and right tangents of $\hat{f}_j$, as illustrated in  Figure~\ref{fig:2d_construction}. In particular, we can see that the radii are given by $\frac{\larc}{\alpha+\beta}$ and $\frac{\larc}{\beta}$, respectively -- independently of the layer index $j$.
    The requirement that the deformation is a rigid body motion for $|x_1| >L/2$ results in the constraint that $2\larc+\max\{l_j\}\le L/2$, as neither the down-slope part nor the arcs can fulfill this condition.
    
    For each $j$, we first define the orientation function $\varphi_j\in W^{1,\infty}(\R)$ by
    \begin{equation}
        \varphi_j(x_1):=
        \begin{cases}
            \frac{x_1}{\larc}\beta, &\text{ if } |x_1|\le \larc,\\
            \beta\sgn{x_1}, &\text{ if } \larc<|x_1|<l_j+\larc,\\
            \beta\sgn{x_1}-\frac{|x_1|-l_j-\larc}{\larc} (\alpha+\beta)\sgn{x_1}, & \text{ if } l_j+\larc\le |x_1|\le l_j+2\larc,\\
            -\alpha\sgn{x_1} ,& \text{ if } |x_1|>  l_j+2\larc,
        \end{cases}
    \end{equation}
    and then define  the
    arc-length preserving map of the $j$-th layer's mid-plane
     $f_j \colon [-L,L] \to \R^2$ by
    \begin{equation}\label{eqdeffjx1}
        f_j(x_1):=\hat f_j(0)+\int_0^{x_1} R_{\varphi_j(x_1')}e_1 dx_1'.
    \end{equation}
    We observe that $f_j\in W^{2,\infty}([-L,L];\R^2)$, with $|f_j'|=1$ and $|f_j''|\le (\alpha+\beta)/\larc$ almost everywhere.
    By comparison with the derivative of $\hat f_j$ as defined in \eqref{eqdeffj} we see that
  $f_j(0)=\hat f_j(0)$, $f_j'(x_1+\larc)=\hat f_j'(x_1)=R_\beta e_1$ for $x_1\in (0, l_j)$ and
    $f_j'(x_1+2\larc)=\hat f_j'(x_1)=R_{-\alpha}e_1$ for $x_1>l_j$.
    A short computation shows that if $x_1\ge \max\{l_j+2\larc,l_{j+1}+2\larc\}$ then
    \begin{equation}\label{eqfjp1fjbc}
    \begin{split}
    f_{j+1}(x_1)-f_j(x_1)
    &=\hat f_{j+1}(0)-\hat f_j(0) + (l_{j+1}-l_j) (R_\beta e_1-R_{-\alpha} e_1)\\
&    = \hat f_{j+1}(x_1)-\hat f_j(x_1)
    =R_{-\alpha} e_2 h_j.
    \end{split}
    \end{equation}
    Further, we observe that
    \begin{equation}
        l_j\le l_0=h \zeta =h\frac{\sin\alpha}{\cos\alpha-\cos\beta}.
    \end{equation}

    \paragraph{Step 3.} 
    We are now ready to define the required mapping $u \colon \omega_h \to \R^2$. We use  the same construction used in 
    \eqref{eqdefufromvplate} in the proof of Lemma \ref{lemmacontinuousbending} with $f_j$ in place of $v$ and set
    \begin{equation} \label{eq:deformation_delam}
    u(x_1,x_2):=f_j(x_1) + (x_2-\delamy_j) (f_j')^\perp(x_1), \quad \text{for $x_2\in [\delamy_j, \delamy_j+h_j)$}.
\end{equation}
This map clearly belongs to $SBV^2_N(\omega_h;\R^2).$

Assume now that  $l_0+2\larc\le \frac12L$. Then for $x_1\ge \frac L2$ we have $f_j'=R_{-\alpha}e_1$, and \eqref{eqfjp1fjbc} holds. Therefore 
$u$ is continuous for $x_1\ge \frac L2$, with
$D u = R_{-\alpha}$. The same holds (flipping some signs) on the other side, and therefore $u$ fulfills the boundary condition \eqref{eqboundarysinglefold2d}.

The energy is estimated by the same argument as in section~\ref{sec:no-delam}, see in particular    \eqref{eqestenergyplate}. 
    In particular, the $n$ individual arcs in the construction have a change in angle of magnitude no more than $2\beta$, an arc-length no less than $\larc$, and a thickness $h_j$. In the rest of the domain the function $f_j$ is affine. Therefore    
    	\begin{equation}\label{eqestenergyplate2}. 
	\begin{split}
\energy_h[u]\le& 
\sum_{j=1}^{n-1}2 \gamma (l_j+2\larc)+
\sum_{j=0}^{n-1}
\int_{-L}^L \int_{b_j}^{b_{j+1}} c | (x_2-b_j) f_j''(x_1) |^2 dx_2 dx_1	\\
\le& 4  n \gamma (h\zeta+\larc)+ c\sum_{j=0}^{n-1} \frac{\beta^2 h_j^3}{\larc}.
	      \end{split}
	\end{equation}

\paragraph{Step 4.}     
We finally show that the map  $u$ defined in \eqref{eq:deformation_delam} is injective.
This leads to an additional constraint. 

We first consider
injectivity inside a single layer. For the affine parts and the arcs around $\pm(l_j+\larc)$ this follows by the same easy argument as in Lemma~\ref{lemmacontinuousbending}. 
For the central arc, with a different concavity, injectivity 
of the expression in \eqref{eq:deformation_delam}
is equivalent to the fact that the layer thickness $h_j$ is not larger than the radius of the arc of circle described by $f_j$, which is $\larc/\beta$. Therefore injectivity is equivalent to
\begin{equation}\label{eqinjectbeta}
 h_j\le \frac{\larc}{\beta} \hskip5mm\text{ for }j=0, \dots, n-1.
\end{equation}

We next focus on the interaction between different layers.
The boundary data automatically imply injectivity for $|x_1|\ge \frac12 L$; we can easily extend the construction to $\R\times [0,h)$ and we see that it suffices to show that  the set
$u(\R\times [b_j, b_j+h_j))$ does not intersect the curve
$(x_1, f_{j+1}(x_1))=u(x_1, b_{j+1})$. 
To ensure global injectivity we therefore need to show
\begin{equation}
    f_j(x_1) + \lambda (f_j')^\perp(x_1) \ne f_{j+1}(x_1')  \text{ for all $j$ and } x_1, x_1'\in\R, \lambda\in[0,h_j).
\end{equation}
Hence, it suffices to prove that
\begin{equation}
    |f_j(s) -f_{j+1}(t)|\ge h_j  \text{ for all $j$ and } s,t.
\end{equation}

\begin{figure}
\begin{adjustwidth}{-2cm}{-2cm}
\centering{
\includegraphics[width=1.0\linewidth]{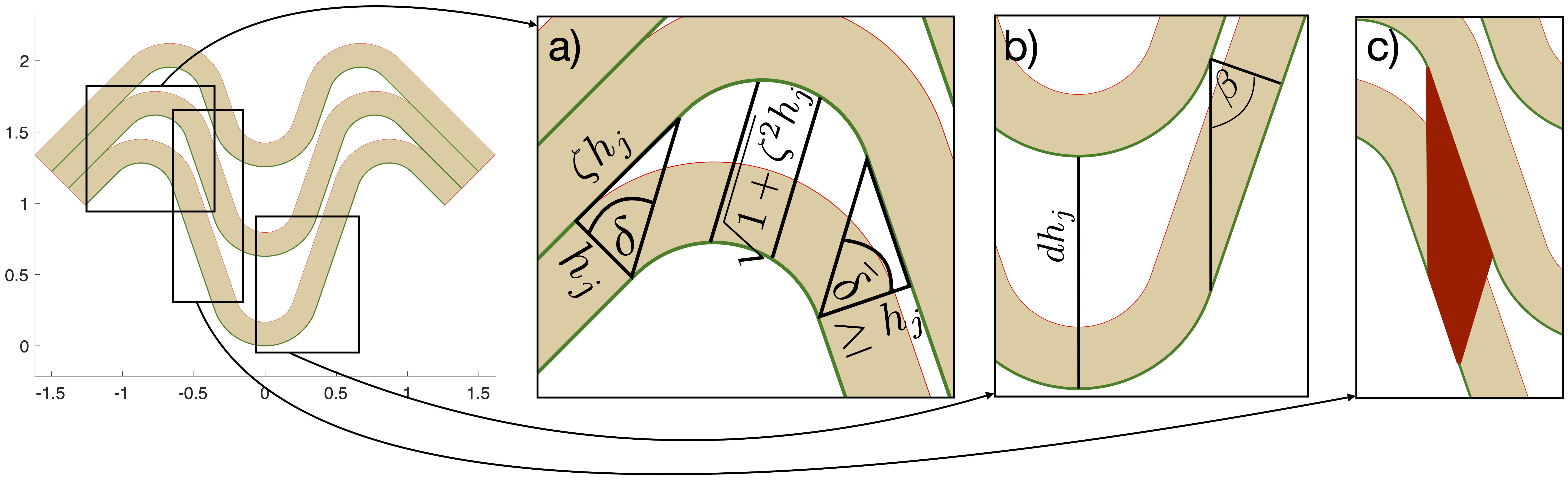}
}
\end{adjustwidth}
\caption{Illustration of injectivity estimates. On the left, the location of subfigures a),b), and c) in the construction is shown. a) relevant trigonometric quantities for the outer part's construction, b) for the inner part, c) in area marked in red, both constructions are valid.} \label{fig:injective}
\end{figure}

We first consider the inner part of the construction. For $|s|,|t|\le \larc+l_{j+1}$, we have $f_{j+1}(t)=f_j(t)+dh_je_2$.
We recall that $|\varphi_j|\le \beta$ everywhere, so that \eqref{eqdeffjx1} implies
\begin{equation}
    |e_2\cdot(f_j(s)-f_j(t))|\le |e_1\cdot(f_j(s)-f_j(t))|\tan\beta.
\end{equation}
Therefore, writing for brevity $A:=f_j(s)-f_j(t)\in\R^2$,
\begin{equation}
\begin{split}
    |f_{j+1}(t)-f_j(s)|^2=& A_1^2+(A_2+dh_j)^2\ge \frac{\cos^2\beta}{\sin^2\beta} A_2^2+ (A_2^2+2A_2dh_j+d^2h_j^2)\\
    =&\frac{1}{\sin^2\beta} A_2^2+2A_2dh_j+d^2h_j^2\\
    \ge& d^2h_j^2-d^2h_j^2{\sin^2\beta}=d^2h_j^2\cos^2\beta\ge h_j^2
    \end{split}
\end{equation}
where in the last step we used the condition $d\cos\beta\ge 1$ which follows from $\beta\le \betaeq(\alpha)$. For an illustration see Figure \ref{fig:injective}b).

For the outer part of the construction now let $|s|,|t| \ge \larc$. We argue similarly as for the inner part and for notational simplicity only consider $t,s<0$, the other side being a symmetric analog. Fix $\delta := \tan^{-1}(\zeta)$ and orthogonal unit vectors $\bar{e}_1:=R_{\alpha-\delta}e_1$, $\bar{e}_2:=R_{\alpha-\delta}e_2 = \begin{pmatrix}
\sin(\delta-\alpha)\\
\cos(\delta-\alpha)
\end{pmatrix}.$ 
From \eqref{eqfjp1fjbc} and $l_j=l_{j+1}+\zeta h_j$ we have, 
for $t+\zeta h_j\le -L/2$,
$f_{j+1}(t+\zeta h_j)-f_j(t)=h_j R_{\alpha}(e_2+\zeta e_1)
=h_j\sqrt{1+\zeta^2} \bar e_2$.
As $f'_{j+1}(t+\zeta h_j)=f'_j(t)$ for $t+\zeta h_j<-\larc$, we conclude that in this range
$f_{j+1}(t+\zeta h_j) = f_{j}(t) + h_j \sqrt{1+\zeta^2} \bar{e}_2$.

We next show that
\begin{equation}\label{eqvarphide}
 |\varphi_j(x_1)-\alpha+\delta|\le \delta \text{ for all } x_1\in [-L,0].
\end{equation}
To see this, note that the angle between the up-slope and $\bar{e}_1$ is exactly equal to $\delta$, where $\cos\delta = \frac{h_j}{\sqrt{1+\zeta^2}h_j}$. The angle $\bar{\delta}$ between the down-slope and $\bar{e}_1$ satisfies 
$\cos \bar{\delta} = \frac{(d\cos\beta) h_j }{\sqrt{1+\zeta^2}h_j} \ge  \cos\delta$, since $\beta\le \betaeq$ and therefore $d\cos\beta\ge 1$. For an illustration, see Figure \ref{fig:injective}a). 
Alternatively, \eqref{eqvarphide} can be proven  algebraically. First, $\varphi_j\in[-\beta,\alpha]$ for $x_1\le0$, so that it suffices to prove that $p:=\frac{\alpha+\beta}2\le\delta$.
We set $m:=\frac{\beta-\alpha}{2}$ and express $\alpha$ and $\beta$ in terms of $m$ and $p$.
As $\tan\delta=\zeta$, by monotonicity of $\tan$ we need to check
\begin{equation}
 \tan p \le \zeta=\frac{\sin\alpha}{\cos\alpha-\cos\beta}=\frac12\left(\frac{1}{\tan m} - 
 \frac{1}{\tan p}\right).
\end{equation}
On the other hand, by \eqref{fabmp} the assumption $\beta\le \betaeq$ (in the form $f_\alpha(\beta)\ge 1$) is the same as
\begin{equation}
 \frac{1}{\tan m} \ge \frac{1+\sin^2 p}{\sin p\cos p} = 2\tan p + \frac{1}{\tan p}.
\end{equation}
As these two conditions are easily seen to be equivalent,  \eqref{eqvarphide}  holds.

Thus we have $|\bar{e}_2\cdot(f_j(s)-f_j(t))|\le |\bar{e}_1\cdot (f_j(s)-f_j(t))|\tan\delta$ for $s,t\le0$. Setting $\bar{A} := \begin{pmatrix}
\bar{e}_1\cdot(f_j(s)-f_j(t))\\
\bar{e}_2\cdot(f_j(s)-f_j(t))
\end{pmatrix}$, we calculate
\begin{equation}\begin{split}
|f_{j+1}(t+h_j\zeta)-f_j(s)|^2 &= \bar{A}_1^2 + (\bar{A}_2+\sqrt{1+\zeta^2}h_j)^2 \\
&\ge \left(\frac{1}{\tan^2\delta} +1\right) \bar{A}_2^2 + 2\bar{A}_2\sqrt{1+\zeta^2}h_j + (1+\zeta^2)h_j^2 \\
&= \frac{1+\zeta^2}{\zeta^2}\bar{A}_2^2 + 2\bar{A}_2\sqrt{1+\zeta^2}h_j + (1+\zeta^2)h_j^2 \\
& \ge h_j^2.
\end{split}\end{equation}

For the mixed case where $s$, $t$ are in different parts of the we note that the minimal distance between two consecutive layers is achieved (again, on the left side of the folding construction) between diagonal corners of the red quadrilateral in Figure \ref{fig:injective}c), as, within this area, both the inner and the outer estimate are valid. It is clear that the respective distances are bounded by $\min\{dh, \sqrt{1+\zeta^2}h\}$.

\begin{lemma}\label{lemmacostrdelaminated}
Fix $\alpha\in (0,\frac\pi4]$, $h>0$, $L \ge h$, $N\in\N$, $N\ge 1$. 
For any $n\in\N$ with $1\le n\le N$, any $\beta\in (\alpha, \betaeq(\alpha)]$, and any $\larc>0$, if
\begin{equation}\label{eqassconstrde}
    \frac{2\beta h}{n} \le \larc\le \frac 18 L \hskip5mm
    \text{ and }\hskip5mm
    \ell:=\frac{\sin\alpha}{\cos\alpha-\cos\beta}h \le\frac{1}{4}L
\end{equation}
then there exists a map $u\in SBV^2_N(\omega_h;\R^2)$, which obeys \eqref{eqboundarysinglefold2d}, is injective,  and such that
\begin{equation}
\energy_h[u]\le c\left(\gamma (\ell+\larc)n + \frac{\beta^2 h^3}{ \larc n^2}\right) = c\left(\gamma n\left(\frac{ h\sin \alpha}{\cos\alpha - \cos\beta}+\larc\right) + \frac{\beta^2 h^3}{\larc n^2 }\right)
\end{equation}
with a constant $c>0$ only depending on the elastic energy density $\Wdd$.
\end{lemma}
\begin{proof}
For $n=1$ there is no delamination, the assertion follows from Lemma~\ref{lemmacontinuousbending}, $\alpha\le\beta$, and $\larc\le L$. Therefore we can assume $n\ge 2$ in the following.

We first choose a subset $\{b_1, \dots, b_{n-1}\}\subseteq \{h\frac 1N, \dots, h\frac{N-1}N\}$ such that, setting $b_0=0$ and $b_{n}=h$, we have $h_j=b_{j+1}-b_j\le \frac{2h}{n}$ for all $j=0,\dots,n-1$. We immediately notice that the map $u$ given in \eqref{eq:deformation_delam} is a member of $SBV^2_N(\omega_h;\R^2)$, with its jump set contained in the set $(-2\larc - \ell,2\larc + \ell)\times\{b_1, \dots, b_{n-1}\}$. The energy estimate follows 
then from \eqref{eqestenergyplate2}.

The conditions in \eqref{eqassconstrde} imply in particular that $2\larc+\ell\le L/2$, and therefore that the boundary condition \eqref{eqboundarysinglefold2d} is fulfilled. The condition \eqref{eqinjectbeta}, required for injectivity around the central arc, follows from 
the first inequality in \eqref{eqassconstrde} and the condition $h_j\le 2h/n$.
\end{proof}

\subsection{Scaling upper bounds.}
\begin{theorem} \label{thm:scaling_upper}
There are $C>0$ and $\eta\in(0,1]$ such that for $0<h\le \eta L$, $\alpha\in (0,\pi/2]$
	there is $u\in SBV^2_N(\Omega;\R^2)$ which obeys \eqref{eqboundarysinglefold2d}, $u\in W^{1,\infty}(\Omega\setminus\overline{J_u};\R^2)$
	with $\calH^1(\overline{J_u}\setminus {J_u})=0$,	
is injective, 
	and obeys
	\begin{equation}
\energy_h[u]\le C \min\{ \frac{\alpha^2h^3}{L}, 
\alpha^{1/3} \gamma^{2/3}h^{4/3} 
       +\frac{\alpha \gamma^{1/2} h^{3/2}}{N^{1/2}}
       + \frac{\alpha^2  h^3}{L N^2 }
       {     + \frac{\alpha  h^4}{L^2 N^2 }}\}.
\end{equation} 
\end{theorem}
\begin{proof}
The result in Theorem \ref{thm:scaling_upper} will follow from Propositions \ref{prop:upper_large_alpha} and \ref{prop:upper_small_alpha}. The first proposition considers the case where the angle $\alpha$ is large, the second treats the case of small $\alpha$. As
$\frac{ h^3}{L N^2 }\le \frac{h^4}{L^2 N^2 }$, in the case of large $\alpha$ the last term does not appear.
\end{proof}
\begin{proposition}\label{prop:upper_large_alpha}
There are $C>0$ and $\eta\in(0,1)$ such that for $0<h\le \eta L$, $\alpha\in (0,\pi/2]$
	there is $u\in SBV^2_N(\Omega;\R^2)$ which obeys \eqref{eqboundarysinglefold2d}, is $W^{1,\infty}$ in $\Omega\setminus\overline{J_u}$, injective, 
	and obeys
\begin{equation}\label{eqprop:upper_large_alpha}
\energy_h[u]\le C \min\{ \frac{h^3}{L}, 
\gamma^{2/3}h^{4/3} 
       +\frac{ \gamma^{1/2} h^{3/2}}{N^{1/2}}
       + \frac{  h^3}{L N^2 }\}.
\end{equation} 
\end{proposition}

\begin{proof}
We prove the proposition in 3 steps.

\emph{Step 1.} 
If $\alpha>\frac\pi4$ we can combine two folds with $L':=L/4$ and $\alpha':=\alpha/2$. Therefore we assume in the remainder of the proof that $\alpha\in(0,\frac\pi4]$.

\emph{Step 2.} By Lemma \ref{lemmacontinuousbending} there is an injective  map $u\in W^{1,\infty}(\omega_h;\R^2)$ with 
$\energy_h[u]\le C  h^3/ L$. This proves the first bound.

\emph{Step 3.} 
We intend to use Lemma \ref{lemmacostrdelaminated}  with
 $\beta:=\betaeq(\alpha)$. 
 We first check the second condition in \eqref{eqassconstrde}.
The function %
\begin{equation}
     \alpha\mapsto \frac{\sin\alpha}{\cos\alpha-\cos\betaeq(\alpha)}
\end{equation}
 is continuous on $(0,\frac\pi4]$, and as $\betaeq = \Oh(\alpha^{1/3})$ for small $\alpha$, 
 it converges to 0 for $\alpha\to0$. Therefore there is $C>0$ such that
\begin{equation}
0<  \frac{\sin\alpha}{\cos\alpha-\cos\betaeq(\alpha)}\le C \hskip5mm
\text{ for all } \alpha\in(0,\frac{\pi}{4}].
\end{equation}
 
By Lemma \ref{lemmacostrdelaminated} we thus know that 
for any $n\in\{1,\dots, N\}$ and $\larc>0$, if 
 \begin{equation} \label{eq:constr1}
    \frac{2\beta h}{n} \le \larc\le \frac 18 L \hskip5mm
    \text{ and }\hskip5mm
    C h \le\frac{1}{4}L
\end{equation}
then 
\begin{equation}\label{eqJu1}
       \inf_u\energy_h[u]\le c\left( \gamma h n+\gamma \larc  n+ \frac{ h^3}{\larc n^2 }\right).
\end{equation}
We assume that
\begin{equation}\label{eqhLadmissibile}
    \frac hL \le \min\{\frac{1}{4C},\frac1{8\pi}\},
\end{equation}
which is guaranteed if $\eta$ is chosen appropriately. Noting that reducing $\larc$ below $\Oh(h)$ does not reduce the energy, due to the first term in \eqref{eqJu1}, we can choose
\begin{equation}
    \larc\in [\pi h, \frac L8].
\end{equation}
Since $\beta\le \frac\pi2$, we are automatically assured that all requirements in \eqref{eq:constr1} are satisfied. At the same time, if $\larc$ is chosen in this range then we can ignore the first term in \eqref{eqJu1}. Condition  \eqref{eqhLadmissibile} ensures that the set of possible values of $\larc$ is nonempty. 

We first optimize the choice of $n$. Set
\begin{equation}
    n:=\min\left\{N,\left\lceil \frac{h}{\gamma^{1/3} \larc^{2/3}}\right\rceil\right\}\in\{1, \dots, N\}.
\end{equation}
Inserting in \eqref{eqJu1} leads to
\begin{equation}
       \inf_u  \energy_h[u]\le 
       \begin{cases}
       \displaystyle c\left(\gamma \larc  + \gamma^{2/3}h\larc^{1/3}\right) ,& \text{ if } h\le N \gamma^{1/3}\larc^{2/3},\\
       \displaystyle c \frac{ h^3}{\larc N^2 }, & \text{ otherwise,}
       \end{cases}
\end{equation}
where the term $\gamma \larc$ stems from the fact the $n\ge 1$.

It remains to choose the value of $\larc$. Since the first expression is strictly increasing, and the second one strictly decreasing, if the critical value is admissible then it is optimal. Precisely, if $\larc^\text{crit} := \gamma^{-1/2} h^{3/2}N^{-3/2}
\in[\pi h, L/8]$ then we obtain a bound of the form $h^{3/2}\gamma^{1/2}/N^{1/2}$.
Otherwise, the relevant constraint is the one that prevents this optimal value from being chosen. Specifically,
\begin{equation}
       \inf_u  \energy_h[u]\le 
       \begin{cases}
       \displaystyle c\left(\gamma h  + \gamma^{2/3}h^{4/3}\right) ,& \text{ if } \gamma^{-1/2} h^{3/2}N^{-3/2}<\pi h,\\[1mm]
       \displaystyle c\frac{\gamma^{1/2}h^{3/2}}{N^{1/2}}, &\text{ if }  \gamma^{-1/2} h^{3/2}N^{-3/2}\in[\pi h, L/8],\\[2mm]
       \displaystyle c \frac{ h^3}{L N^2 }, & \text{ if } L< 8 \gamma^{-1/2}h^{3/2}N^{-3/2},
       \end{cases}
\end{equation}
or, equivalently,
\begin{equation}\label{eqenfinalalar}
       \inf_u  \energy_h[u]\le 
       \begin{cases}
       \displaystyle c\left(\gamma h  + \gamma^{2/3}h^{4/3}\right) ,& \text{ if } h<\gamma N^3,\\[1mm]
       \displaystyle c\frac{\gamma^{1/2}h^{3/2}}{N^{1/2}}, &\text{ if }   \gamma N^3\le h\le \gamma^{1/3}L^{2/3}N,\\[2mm]
       \displaystyle c \frac{ h^3}{L N^2 }, & \text{ if } \gamma^{1/3}L^{2/3}N< h.
       \end{cases}
\end{equation}
We remark that the first and the last regime are disjoint, since 
$\gamma^{1/3}L^{2/3}N<h$ and $h\le L$ imply
$\gamma^{1/3}h^{2/3}N< h$, which is equivalent to
$\gamma N^3< h$.

The calculation above thus reveals energetic regimes as follows, which differ in the number of delaminated layers and the delamination length.
\begin{description}
\item{Sharp fold partial delamination:}
       The energy scaling $\gamma h  + \gamma^{2/3}h^{4/3}$ has the shortest possible delamination length $\larc=\pi h$ and $n\sim \lceil \gamma^{-1/3}h^{1/3}\rceil$. 
       It originates by balancing the two terms 
$\gamma \larc  n+ \frac{ h^3}{\larc n^2 }$ after setting $\larc=h$.

The term $\gamma h$, corresponding to $n=1$, can be dropped. To see this, we first note that it is only relevant if $h\le\gamma$. In this regime, however, it would be convenient not to delaminate at all, obtaining an energy $h^3 L^{-1}$. Indeed, as $h\le L$, $h\le\gamma$ implies $h^3 L^{-1} \le h^2 \le \gamma h$.

\item{Localized full delamination:} The energy scales as 
$\frac{\gamma^{1/2}h^{3/2}}{N^{1/2}}$. As $n=N$, each layer is delaminated, however only over a length $\larc=\gamma^{-1/2}h^{3/2}N^{-3/2}$.
This originates by balancing the two terms 
$\gamma \larc  n+ \frac{ h^3}{\larc n^2 }$ after setting $n=N$. All layers are delaminated, but only over a length $\larc$. The energy corresponds to $N$ plates of thickness $h/N$ bent over a length $\larc$, the value is determined so that this exactly balances the delamination energy.

\item{Total delamination:} The energy scales as $\frac{ h^3}{L N^2 }$, there are $N$ separate plates of thickness $h/N$, each bending over the entire available length $L$, the delamination energy is smaller.
\end{description}

Since the powers of $h$ are increasing, and the expression is continuous, one can also write the above regimes concisely as
in \eqref{eqprop:upper_large_alpha}.
\end{proof}

\begin{proposition} \label{prop:upper_small_alpha}
There are $C>0$, $\alpha_0\in(0,\pi/2]$ such that for $0<h\le L$, $\alpha\in (0,\alpha_0]$
	there is $u\in SBV^2_N(\Omega;\R^2)$ which obeys \eqref{eqboundarysinglefold2d}, is $W^{1,\infty}$ in $\Omega\setminus\overline{J_u}$, injective, 
	and obeys
\begin{equation}
\energy_h[u]\le C \min\{ \frac{\alpha^2 h^3}{L}, 
\alpha^{1/3} \gamma^{2/3}h^{4/3} 
       +\frac{\alpha \gamma^{1/2} h^{3/2}}{N^{1/2}}
       + \frac{\alpha^2  h^3}{L N^2 }
      + \frac{\alpha  h^4}{L^2 N^2 }
       \}.
\end{equation} 
\end{proposition}

\begin{proof} As above, we prove the two bounds separately. \\
\emph{Step 1.} Again, we use Lemma \ref{lemmacontinuousbending} to obtain an injective map $u\in W^{1,\infty}(\omega_h;\R^2)$ with 
$\energy_h[u]\le C  \alpha^2h^3/ L$.

\emph{Step 2.}
The delaminated construction is obtained with Lemma \ref{lemmacostrdelaminated} with a careful choice of the parameters.
In this regime of small $\alpha$, any delaminated construction must also satisfy that $\beta$ is small, as $\beta \le \betaeq(\alpha)\simeq (4\alpha)^{1/3}$ for $\alpha\to0$. A straightforward Taylor series expansion of the delamination length factor $\zeta$ leads to
\begin{equation}
   \zeta =  \frac{\sin\alpha}{\cos\alpha-\cos\beta}= \frac{2\alpha + O(\alpha^3)} {\beta^2-\alpha^2+ O(\beta^4)+O(\alpha^4)}.
\end{equation}
Since $\alpha\le \beta$, the $O(\alpha^4)$ in the denominator can be ignored. It is however important to avoid that $\beta^2-\alpha^2$ in the denominator becomes too small.
For this reason we shall restrict the choice of $\beta$ by assuming $\beta\ge 2\alpha$, which implies
$\cos\alpha-\cos\beta\ge \cos\frac\beta2-\cos\beta=\frac38 \beta^2+O(\beta^4)$. For $\beta$ sufficiently small, this is larger than $\beta^2/3$. Therefore
there is $\alpha_0>0$ such that 
for all $\alpha\in (0,\alpha_0]$  one has
$2\alpha\le\betaeq(\alpha)$ and
for all $\beta\in [2\alpha,\betaeq(\alpha)]$
\begin{equation}
  \zeta= \frac{\sin\alpha}{\cos\alpha-\cos\beta}\le  3\frac{\alpha } {\beta^2}.
\end{equation}
By Lemma \ref{lemmacostrdelaminated} we know that 
for any $n\in\{1,\dots, N\}$ and $\larc>0$, if 
 \begin{equation} \label{eq:admissible2}
    \frac{2\beta h}{n} \le \larc \le \frac 18 L \hskip5mm
    \text{ and }\hskip5mm
    3  \frac{\alpha}{\beta^2}h \le\frac{1}{4}L
\end{equation}
(noting that $h\le L$ guarantees that the set of admissible $\beta\le\betaeq$ is not empty, after potentially decreasing $\alpha_0$ again) then 
\begin{equation}\label{eqJu12}
       \inf_u\energy_h[u]\le c\left( \frac{\alpha}{\beta^2}\gamma h n +\gamma \larc  n+ \frac{\beta^2 h^3}{\larc n^2 }\right).
\end{equation}
We assume that
\begin{equation} \label{eq:admissible3}
    \frac{\alpha}{\beta^2}h\le \larc\le \frac18 L
\end{equation}
so that the first term in the energy can be ignored.
We next compare the first lower bound in \eqref{eq:admissible2} with the one in \eqref{eq:admissible3}. 
Since we know that 
$\beta\le \betaeq(\alpha)\le c \alpha^{1/3}$, and $n\ge 1$, we obtain for some $c_*>0$
\begin{equation}
     \frac{2\beta h}{n}\le c_*\frac{\alpha}{\beta^2}h.
\end{equation}
Without loss of generality we can assume $c_*\ge\frac32$.
We then see that all conditions in \eqref{eq:admissible2} and in  \eqref{eq:admissible3} are satisfied provided
\begin{equation}\label{eq:admissible23}
     c_*\frac{\alpha}{\beta^2}h\le \larc\le \frac18 L.
\end{equation}
By the same argument we see that this condition is optimal (up to a factor).
We choose
\begin{equation}
    n:=\min\left\{N,\left\lceil \frac{\beta^{2/3}h}{\gamma^{1/3} \larc^{2/3}}\right\rceil\right\}\in\{1, \dots, N\}.
\end{equation}
Inserting in \eqref{eqJu12} leads to
\begin{equation}
       \inf_u  \energy_h[u]\le 
       \begin{cases}
       \displaystyle c\left(\gamma \larc  +  \beta^{2/3} \gamma^{2/3} h \larc^{1/3}\right) ,& \text{ if } \beta^{2/3}h\le  \gamma^{1/3}\larc^{2/3} N,\\[2mm]
       \displaystyle c \frac{ \beta^2 h^3}{\larc N^2 }, & \text{ otherwise,}
       \end{cases}
\end{equation}
noting again that the term $\gamma \larc$ stems from rounding $n$ up to the next integer, which is needed as we require $n\ge1$. As in Step~3 of the proof of Proposition~\ref{prop:upper_large_alpha}, the first expression is increasing in $\larc$, the second decreasing. Therefore the optimal value is the critical one, $\larc^\text{crit}:=\beta \gamma^{-1/2} h^{3/2}N^{-3/2}$, if it is admissible in the sense of \eqref{eq:admissible23}. If not, the optimal value is the admissible value closest to the critical one.

Assume now that $\larc=\larc^\text{crit}$ is admissible. Then one obtains $n=N$, and 
$\inf_u  \energy_h[u]\le 
\displaystyle c\frac{\beta \gamma^{1/2}h^{3/2}}{N^{1/2}}$. Restating the requirements on $\larc$ from 
\eqref{eq:admissible23}, this is possible if 
$\larc^\text{crit}\ge c_*\alpha h/\beta^{2}$ and
$\larc^\text{crit}\le L/8$, which is the same as 
\begin{equation}
h \ge c_*^2\alpha^2\beta^{-6}\gamma N^3 \label{eq:lower_h2}  \end{equation}
and
\begin{equation}
h \le \frac14 \beta^{-2/3}\gamma^{1/3}L^{2/3} N \label{eq:upper_h}.
\end{equation}

Assume now that one of these two conditions is violated
(one can easily check that if  $c_*\frac{\alpha}{\beta^2}h\le \frac18 L$, so that \eqref{eq:admissible23} gives a nonempty set of admissible values of $\larc$, then it is not possible for both 
\eqref{eq:lower_h2} and \eqref{eq:upper_h}
to be violated at the same time).
\begin{itemize}
\item If \eqref{eq:lower_h2} is violated, then 
$\larc=c_*  \alpha \beta^{-2}h$,
$n= \left\lceil \frac{\beta^{2/3}h}{\gamma^{1/3} \larc^{2/3}}\right\rceil$,
and
$\inf_u  \energy_h[u]\le c( \alpha\beta^{-2}\gamma h +\alpha^{1/3}\gamma^{2/3} h^{4/3})$.
\item If  \eqref{eq:upper_h} is violated, then $\larc=L/8$, $n=N$, and 
$\inf_u  \energy_h[u]\le \displaystyle c \frac{ \beta^2h^3}{L N^2 }$.
\end{itemize}

Summarizing, this leads to
\begin{equation}
E:=       \inf_u  \energy_h[u]\le 
       \begin{cases}
       \displaystyle c\left(\frac{\alpha\gamma h}{\beta^{2} } + \alpha^{1/3}\gamma^{2/3}h^{4/3}\right) ,& \text{ if }  h<c_*^2\alpha^2\beta^{-6}\gamma N^3  ,\\[2mm]
       \displaystyle 
       c\frac{\beta \gamma^{1/2}h^{3/2}}{N^{1/2}}, &\text{ if } 
      c_*^2\alpha^2\beta^{-6}\gamma N^3 \le h \le \frac14 \beta^{-2/3}\gamma^{1/3}L^{2/3} N,\\[2mm]
       \displaystyle c \frac{ \beta^2h^3}{L N^2 }, & \text{ if } 
       h> \frac14 \beta^{-2/3}\gamma^{1/3}L^{2/3} N,
       \end{cases}
\end{equation}
or, equivalently,
\begin{equation}\label{eqinfenlbs3}
       E \le 
       \begin{cases}
       \displaystyle c\left( \frac{\alpha\gamma h}{\beta^{2} }  + \alpha^{1/3} \gamma^{2/3}h^{4/3}\right) ,& \text{ if } 
       \beta<\alpha^{1/3}\gamma^{1/6}h^{-1/6}N^{1/2},\\[2mm]
       \displaystyle c\frac{\beta \gamma^{1/2}h^{3/2}}{N^{1/2}}, &\text{ if }  
       \alpha^{1/3}\gamma^{1/6}h^{-1/6}N^{1/2}\le\beta\le 
       \gamma^{1/2} h^{-3/2}L N^{3/2},
       \\[2mm]
       \displaystyle c \frac{ \beta^2h^3}{L N^2 }, & \text{ if }
       \beta> \gamma^{1/2}h^{-3/2}LN^{3/2}.
       \end{cases}
\end{equation}
It remains to choose $\beta\in[2\alpha,\betaeq(\alpha)]$
with $c_*\alpha\beta^{-2}h\le \frac18 L$
 (from \eqref{eq:admissible23}), 
 i.e.,
\begin{equation}\label{eqcondbeta}
\max\{2\alpha, (8c_*\alpha h/L)^{1/2}\}\le\beta\le \betaeq(\alpha). 
\end{equation}
 These bounds scale as $\alpha+\alpha^{1/2}(h/L)^{1/2}$ and $\alpha^{1/3}$, respectively, and in particular, possibly reducing $\alpha_0$, we can ensure that the set is not empty for all values of the parameters.

 We observe that the expression in \eqref{eqinfenlbs3} is minimized at $\betacrit:=\alpha^{1/3}\gamma^{1/6}h^{-1/6}N^{1/2}$, and that -- provided $\alpha_0$ is chosen sufficiently small -- $\alpha^{1/3}\le\betaeq(\alpha)$ for all $\alpha\in(0,\alpha_0]$. There are different regimes, both of which are characterized by the delamination length being at the minimum value still satisfying the injectivity conditions. 
In these two regimes the optimal scaling is attained by the choice $\beta=\betaeq(\alpha)$, which corresponds to the fact that the layers touch each other, as illustrated in Figure~\ref{fig:2d_construction}.
The first one delaminates some of the layers, the second one all layers.
 \begin{description}
\item{Sharp fold partial delamination:} If $\betacrit>\alpha^{1/3}$, which is the same as $h<\gamma N^3$,  then we take $\beta=\alpha^{1/3}$, and we are in the first regime in \eqref{eqinfenlbs3}. This gives $E\le c\alpha^{1/3}\gamma h + c\alpha^{1/3} \gamma^{2/3}h^{4/3}$.
The first term does not contribute. Indeed, if $ \alpha^{1/3} \gamma^{2/3}h^{4/3}\le \alpha^2 h^3/L$ then necessarily $\gamma^{2/3}L\le \alpha^{5/3}h^{5/3}$, which, since $h\le L$ and $\alpha\le 1$, implies $\gamma\le h$. Therefore we can drop the $\alpha^{1/3}\gamma h$ term and obtain 
$E\le c\min\{\alpha^{1/3} \gamma^{2/3}h^{4/3},\alpha^2 h^3/L\}$. In this regime  $\larc= \alpha^{1/3}h$ and $n\simeq \gamma^{-1/3}h^{1/3}\le N$, so that the energy bound
behaves as $\gamma\larc n$.

\item{Sharp fold full delamination:}
If $
\max\{2\alpha, (8c_*\alpha h/L)^{1/2}\}\le \betacrit\le
\alpha^{1/3}$, then $\beta=\betacrit$ is optimal and 
$E\le c\alpha^{1/3}\gamma^{2/3}h^{4/3}$.
In this regime $n=N$ and $\larc=\alpha^{1/3}\gamma^{-1/3}h^{4/3} N^{-1}$ (indeed, the bound on $E$ is of the form $\gamma \larc N$).
\end{description}
If $
\betacrit <\max\{2\alpha, (8c_*\alpha h/L)^{1/2}\}$
then we take $\beta=\max\{2\alpha, (8c_*\alpha h/L)^{1/2}\}$, and we are in the second or third regime in \eqref{eqinfenlbs3}. This gives in principle four subcases, one of which cannot occur. The first of the three possible cases still delaminates only locally, the other two are delaminated over the entire length of the specimen. In any case, all layers are delaminated.
\begin{description}
\item{Localized full delamination:} 
Assume first that $\alpha\ge 2c_* h/L$, 
so that $\beta=2\alpha$.
If 
$2\alpha\le\gamma^{1/2}h^{-3/2}LN^{3/2}$ then 
we are in the second regime in \eqref{eqinfenlbs3} (with $n=N$ and $\larc=\larc^\text{crit}=2\alpha \gamma^{-1/2} h^{3/2}N^{-3/2}$) and
$E\le c\frac{\alpha \gamma^{1/2}h^{3/2}}{N^{1/2}}$.

\item{Total delamination:} If 
$\alpha\ge 2c_* h/L$ and
$2\alpha>\gamma^{1/2}h^{-3/2} L N^{3/2}$, then 
we are in the third regime in \eqref{eqinfenlbs3} (with $n=N$ and $\larc= L/8$) and
$E\le c\frac{ \alpha^2h^3}{L N^2 }$.

\item{Small-angle total delamination:}
If instead 
$\alpha< 2c_* h/L$ then $\beta=(8c_*\alpha h/L)^{1/2}$. 
A short computation shows that $\betacrit\le (\alpha h/L)^{1/2}$ is equivalent to $(\alpha h/L)^{1/2}\ge \gamma^{1/2}h^{-3/2}LN^{3/2}$, so that we are necessarily in the third regime in \eqref{eqinfenlbs3} (with $n=N$ and $\larc= L/8$)  and $E\le c \frac{\alpha h^4}{ L^{2}N^{2}}$.
\end{description}

In all cases we may conclude
\begin{equation}
\inf_u  \energy_h[u]\le 
       c\min\left\{
       \frac{\alpha^2h^3}{L},
       \alpha^{1/3}\gamma^{2/3}h^{4/3} 
       +\frac{\alpha \gamma^{1/2} h^{3/2}}{N^{1/2}}
       {+ \frac{\alpha  h^4}{L^2 N^2 }}
       + \frac{\alpha^2 h^3}{L N^2 }
       \right\}.
\end{equation}
\end{proof}

\begin{remark}\label{eqmarkhsmall}
A more precise summary of the cases above shows that 
for $h\le \gamma N^3$, then only the ``localized partial delamination'' regime is relevant, and one obtains
(for $\alpha\in(0,\alpha_0]$)
 \begin{equation}
\inf_u  \energy_h[u]\le 
       c\min\left\{
       \frac{\alpha^2h^3}{L},
       \alpha^{1/3}\gamma^{2/3}h^{4/3} 
       \right\}.
\end{equation}
Recalling \eqref{eqenfinalalar} one can see that this holds for all $\alpha\in(0,\frac\pi2]$. As the delamination length is proportional to $\alpha^{1/3}h$, it is always smaller than $L$. In turn, the number of delaminated layes $n\simeq \gamma^{-1/3}h^{1/3}$ does not depend on the opening $\alpha$. Both the delamination length and the number of delaminated layers are discontinuous at yielding point $\alpha\simeq \gamma^{2/5}L^{3/5}/h$.
\end{remark}

\begin{remark}\label{remarkubfll}
The statement in Theorem \ref{thm:scaling_upper} (analogously in Propositions \ref{prop:upper_large_alpha} and \ref{prop:upper_small_alpha}), can also be written in the form
\begin{equation}
\energy_h[u]\le C \min\left\{ \frac{\alpha^2h^3}{L}, 
\max\bigg\{\alpha^{1/3} \gamma^{2/3}h^{4/3} 
       ,\frac{\alpha \gamma^{1/2} h^{3/2}}{N^{1/2}}
       ,   {  \frac{\alpha  h^4}{L^2 N^2 }}
       ,\frac{\alpha^2  h^3}{L N^2 }
       \bigg\}\right\}.
\end{equation}
To see this, it suffices to observe that $\max\{a,b\}\sim a+b$ for any $a,b\ge0$.

In order to gain some understanding in these many regimes we sort them in increasing values of $\alpha$
(see also Figure~\ref{fig:energy_moment} below). For definiteness let us assume that the first linear term is the relevant  one, in the sense that 
$\gamma^{1/2}h^{3/2}N^{-1/2}\ge h^4L^{-2}N^{-2}$, so that the last term in the maximum can be ignored. 
The condition is equivalent to $h^5\le \gamma L^4 N^3$.
A simple computation shows that the maximum is then equal to
\begin{equation}
\begin{cases}
\alpha^{1/3} \gamma^{2/3}h^{4/3}, &\text{ if }
\alpha< \gamma^{1/4}N^{3/4}/h^{1/4},\\
\frac{\alpha \gamma^{1/2}h^{3/2}}{N^{1/2}},&\text{ if }
\gamma^{1/4}N^{3/4}/h^{1/4}\le\alpha\le  
\gamma^{1/2}LN^{3/2} /h^{3/2},\\
\frac{\alpha^2  h^3}{L N^2 },&\text{ if }
\alpha> \gamma^{1/2} L N^{3/2} /h^{3/2}.
\end{cases}
\end{equation}
It remains to take the minimum between this expression and $\alpha^2h^3/L$. This is always larger than the third expression, and one easily computes the points where this intersects the other two. There are two cases. If 
$\gamma^{2/5}L^{3/5}/h< \gamma^{1/4}N^{3/4}/h^{1/4}$ (which is the same as $\gamma L^4< h^5N^5$) then we obtain
\begin{equation}\label{eqecasessortedinallpha}
\frac1c\energy_h[u]\le 
\begin{cases}
\frac{\alpha^2h^3}{L}, & \text{ if } \alpha<\gamma^{2/5}L^{3/5}/h,\\
\alpha^{1/3} \gamma^{2/3}h^{4/3}, &\text{ if }
\gamma^{2/5}L^{3/5}/h\le\alpha< \gamma^{1/4}N^{3/4}/h^{1/4},\\
\frac{\alpha \gamma^{1/2} h^{3/2}}{N^{1/2}},&\text{ if }
\gamma^{1/4}N^{3/4}/h^{1/4}\le\alpha< \gamma^{1/2} L N^{3/2} /h^{3/2},\\
\frac{\alpha^2  h^3}{L N^2 },&\text{ if }
\alpha\ge \gamma^{1/2}L N^{3/2} /h^{3/2}.
\end{cases}
\end{equation}
whereas for 
$\gamma^{2/5}L^{3/5}/h\ge  \gamma^{1/4}N^{3/4}/h^{1/4}$ the $\alpha^{1/3}$ regime disappears and
\begin{equation}
\frac1c\energy_h[u]\le 
\begin{cases}
\frac{\alpha^2h^3}{L}, & \text{ if } \alpha<\gamma^{1/2}LN^{-1/2}/ h^{3/2},\\
\frac{\alpha \gamma^{1/2}h^{3/2}}{N^{1/2}},&\text{ if }
\gamma^{1/2}LN^{-1/2}/h^{3/2}\le\alpha< \gamma^{1/2} LN^{3/2} /h^{3/2},\\
\frac{\alpha^2  h^3}{L N^2 },&\text{ if }
\alpha\ge \gamma^{1/2}L N^{3/2}  /h^{3/2}.
\end{cases}
\end{equation}
\end{remark}

\begin{remark}\label{remarkformell}
 It is interesting to compute the optimal delaminated length $\ell$ in the different regimes. Collecting the results from the above computations we obtain (in the setting of \eqref{eqecasessortedinallpha} and with $h\ge\gamma N^3$) the scaling
 \begin{equation}\label{eqscalingell}
 \ell \approx 
\begin{cases}
0 & \text{ if } \alpha<\gamma^{2/5}L^{3/5}/h,\\
\alpha^{1/3}\gamma^{-1/3}h^{4/3} N^{-1} &\text{ if } \gamma^{2/5}L^{3/5}/h\le\alpha< \gamma^{1/4}h^{-1/4}N^{3/4},\\
\alpha \gamma^{-1/2} h^{3/2}N^{-3/2}&\text{ if }
\gamma^{1/4}h^{-1/4}N^{3/4}\le\alpha<  \gamma^{1/2}h^{-3/2} L N^{3/2},\\
L ,&\text{ if }
\alpha\ge \gamma^{1/2}h^{-3/2} L N^{3/2} .
\end{cases}
\end{equation}
 \end{remark}

\section{Lower bound} \label{sec:lower}

\begin{theorem}\label{theolowerbound}
Let $\alpha\in(0,\frac\pi2]$, $h\le \frac14 L$. 
Assume that $u\in SBV^2_N(\omega_h;\R^2)$ obeys the boundary condition \eqref{eqboundarysinglefold2d} and has
jump set contained in the product of an interval and a finite set, 
\begin{equation}\label{eqJuintervaltheo}
  J_u\subseteq (x_-,x_+)\times\{b_1, \dots b_{n-1}\}
 \end{equation}
(up to $\calH^1$-null sets)
 for some $n\ge 1$ and $x_-,x_+$ with $-\frac L2\le x_-\le x_+\le \frac L2$.
Assume additionally
\begin{equation}\label{eqcalhjcjn}
    \calH^1(J_u) \ge c_J (n-1)(x_+-x_-)
\end{equation}
for some $c_J>0$.
Then
\begin{equation}
\mathcal \energy_h[u]\ge
  c\min\left\{\frac{\alpha^2h^3}{L},
 \frac{\alpha \gamma^{1/2} h^{3/2} }{N^{1/2}}+ \frac{\alpha^2h^3}{LN^2},
\frac{\alpha^2h^2}{N}\right\},
\end{equation}
where $c$ may depend on $c_J$.
\end{theorem}
As in the construction of the upper bound, $n\ge1$ denotes the number of layers, $n-1\ge0$ the number of interfaces.
\begin{remark}
We note that the all regimes obtained here, except for the last one, are matched by upper bounds. In the proof below, we will see that this last regime corresponds to a very short delamination length. In the upper bound, this is forbidden by the injectivity constraint $\larc \ge \frac{2\beta h}{n}$, which however only arises from our specific construction. Similarly, the upper bounds $\energy_h[u] \le c\alpha^{1/3}\gamma^{2/3}h^{4/3}$, arising  due to the above constraint $\larc \ge \frac{2\beta h}{n}$, and $\energy_h[u] \le c\alpha h^4 L^{-2}N^{-2}$, arising due to the constraint $\zeta \le L/4$ in our construction, are not matched by a lower bound.
\end{remark}

\begin{lemma}\label{lemmalowerbnofracture}
Let $\alpha\in(0,\frac\pi2]$, $h\le \frac14 L$. 
Assume that $u\in W^{1,2}(\omega_h;\R^2)$ obeys the boundary condition \eqref{eqboundarysinglefold2d}. Then 
\begin{equation}
 \energy_h[u]\ge c\frac{\alpha^2h^3}{L}.
\end{equation} 
\end{lemma}
This result follows easily from the compactness result for plates in \cite{FrieseckeJamesMueller2002}; we present here the short argument since it will be reused in the following.
\begin{proof}
We consider the rectangles $Q_k:=(hk, h(k+2))\times (0,h)$, for 
 $k=0, \dots, K:=\lfloor \frac{L}{h}-2\rfloor$.
 By \eqref{eqW2dbounds} and the Friesecke-James-Müller rigidity estimate \cite{FrieseckeJamesMueller2002}, there is a universal constant $c>0$ such that for every $k$ there is $R^k\in\SO(2)$ with
 \begin{equation}
  \int_{Q_k}|Du-R^k|^2 dx \le c 
  \int_{Q_k}\dist^2(Du,\SO(2))dx\le c
  \int_{Q_k}\Wdd(Du) dx.
 \end{equation}
 By \eqref{eqboundarysinglefold2d}, we can assume $R^0=R_{\alpha}$ and  $R^K=R_{-\alpha}$. 
With a triangular inequality and $\calL^2(Q_k\cap Q_{k+1})=h^2$ we obtain
 \begin{equation}
  h^2 |R^k-R^{k+1}|^2\le 2 \int_{Q_k}|Du-R^k|^2 dx +
    2 \int_{Q_{k+1}}|Du-R^{k+1}|^2 dx .
 \end{equation}
By a triangular inequality and Hölder's inequality,
\begin{equation}
 |R^0-R^K|^2\le \left(\sum_{k=0}^{K-1} |R^k-R^{k+1}|\right)^2 \le K \sum_{k=0}^{K-1} |R^k-R^{k+1}|^2.
\end{equation}
Combining the previous estimates and
$K\le L/h$,
\begin{equation}
|R^0-R^K|^2 \le c \frac{K}{h^2} \int_{\omega_h} \Wdd(Du) dx 
\le c\frac{L}{h^3} \int_{\omega_h} \Wdd(Du) dx .
\end{equation}
With $|R_\alpha-R_{-\alpha}|\ge 2\sin\alpha\ge \alpha$ the proof is concluded.
\end{proof}

\begin{lemma}\label{lemmalowerbound2}
Let $\alpha\in(0,\frac\pi2]$, $h\le \frac14 L$. 
Assume that $u\in SBV^2_N(\omega_h;\R^2)$ obeys the boundary condition \eqref{eqboundarysinglefold2d} and has
jump set contained in the product of an interval and a finite set, 
\begin{equation}\label{eqJuinterval}
  J_u\subseteq (x_-,x_+)\times\{b_1, \dots b_{n-1}\}
 \end{equation}
 for some $n\ge 1$ and $x_-,x_+$ with $-\frac L2\le x_-< x_+\le \frac L2$.
Then, with $\ell:=x_+-x_-$,
\begin{equation}\label{eqlowerbsecresd}
 \int_{\omega_h} \Wdd(Du)
dx
 \ge c\min\left\{\frac{\alpha^2h^3}{L},
\frac{\alpha^2h^3}{n^2(\ell+h/n)}\right\}.
\end{equation}
\end{lemma}
\begin{proof}
By \eqref{eqJuinterval}, $u\in W^{1,2}((-L,x_-)\times(0,h))$. This set can be treated as 
 in Lemma \ref{lemmalowerbnofracture}. 
We consider the rectangles
$Q^-_k:=(x_--(k+2)h, x_--kh)\times(0,h)$, $k=0, \dots, K^-:=\lfloor
\frac{x_-+L}{h}-2\rfloor$,
and obtain rotations $R^k_-\in \SO(2)$ such that $R_-^{K^-}=R_{\alpha}$,
\begin{equation}
 |R_-^0-R_\alpha|^2\le K^-\sum_{k=0}^{K^--1} |R_-^k-R_-^{k+1}|^2\le c\frac{L}{h^3} \int_{\omega_h} \Wdd(Du) dx.
\end{equation}
Analogously, \eqref{eqJuinterval} also gives $u\in W^{1,2}((x_+,L)\times(0,h))$. 
The same computation, using the sets
$Q^+_k:=(x_++kh, x_++(k+2)h)\times(0,h)$, leads to
\begin{equation}\label{eqrp0r0}
 |R_+^0-R_{-\alpha}|^2\le K^+\sum_{k=0}^{K^+-1} |R_+^k-R_+^{k+1}|^2\le c\frac{L}{h^3} \int_{\omega_h} \Wdd(Du) dx .
\end{equation}
By a triangular inequality, 
\begin{equation}
 |R_+^0-R_{-\alpha}|+
 |R_+^0-R_-^0|+
 |R_-^0-R_{\alpha}|\ge
 |R_\alpha-R_{-\alpha}|\ge 2\sin\alpha\ge \alpha.
\end{equation}
If $\max\{|R^0_+-R_{-\alpha}|, 
|R^0_--R_{\alpha}|\}\ge \frac 14\alpha$, then 
\eqref{eqlowerbsecresd} holds and the proof is concluded. Therefore for the rest of the proof we can assume
\begin{equation}\label{eqr0mr0p}
|R^0_--R^0_+|\ge \frac 12\alpha,
\end{equation}
and we only need to deal with the central part of the domain.
As in the proof of the upper bound we set $b_0:=0$, $b_{n+1}:=h$, and $h_j:=b_{j+1}-b_j$. 
We treat each layer $(x_-,x_+)\times (b_j, b_{j+1})$ separately,
for $j=0,\dots, n$. We consider the sets
$Q_k^j:=(x_-+(k-1)h_j, x_-+(k+1)h_j)\times(b_j, b_j+h_j)$, for 
$k=0, \dots, K_j:=\lfloor \frac{x_+-x_-}{h_j}+1\rfloor$, and obtain
rotations $R_j^k\in\SO(2)$ with
\begin{equation}\label{eqinnerchain}
 |R_j^0-R_j^{K_j}|^2\le c \frac{K_j}{h_j^2} 
 \int_{(-L,L)\times(b_j, b_{j+1})} \Wdd(Du) dx 
\end{equation}
as well as
\begin{equation}
 \int_{Q^j_0} |Du-R_j^0|^2 dx +
\int_{Q^j_{K_j}} |Du-R_j^{K_j}|^2 dx \le c 
\int_{Q^j_0\cup Q^j_{K_j}} \Wdd(Du)dx.
 \end{equation}
Since $\calL^2(Q^j_0\cap Q^-_0)\ge h_j^2$ and
$\calL^2(Q^j_{K_j}\cap Q^+_0)\ge h_j^2$, 
\begin{equation}
 h_j^2|R^0_j-R^0_-|^2 \le c \int_{Q^j_0} \Wdd(Du) dx +
 c\int_{Q^j_0\cap Q^-_0} |Du-R^0_-|^2dx,
\end{equation}
and the same on the other side. With  \eqref{eqinnerchain} we obtain
\begin{equation}\label{eqestonej}
\begin{split}
 h_j^2|R^0_--R^0_+|^2 \le &c \int_{Q^j_0\cup Q^j_{K_j}} \Wdd(Du) dx +
 c\int_{Q^j_0\cap Q^-_0} |Du-R^0_-|^2dx
 +
 c\int_{Q^j_{K_j}\cap Q^+_0} |Du-R^0_+|^2dx
 \\
& +c K_j
 \int_{(-L,L)\times(b_j, b_{j+1})} \Wdd(Du) dx .
 \end{split}
\end{equation}
Let $A:=\{j\in\{1,\dots, n-1\}: h_j\ge \frac h{2n}\}$.  Then $\sum_{j\not\in A} h_j\le \frac h2$, which implies
\begin{equation}
 \sum_{j\in A} h_j^2 \ge \frac{1}{\# A}(\sum_{j\in A} h_j)^2 \ge \frac{h^2}{4n}.
\end{equation}
We sum \eqref{eqestonej} over all $j\in A$, use that the domains of integration for different $j$ are disjoint and  that $j\in A$ implies $K_j\le 1+\ell/h_j\le 1+2n\ell/h\le 2(h+n\ell)/h$, and obtain
\begin{equation}
 \frac{h^2}{4n} 
 |R^0_--R^0_+|^2 \le
 \sum_{j\in A} h_j^2 |R^0_--R^0_+|^2 
 \le c (1+\frac{n\ell}{h})\int_{\omega_h} \Wdd(Du) dx.
\end{equation}
Recalling \eqref{eqr0mr0p}, 
\begin{equation}
 \int_{\omega_h} \Wdd(Du) dx\ge c \frac{\alpha^2 h^3}{n^2(\ell+h/n)}
\end{equation}
which concludes the proof.
\end{proof}

\begin{proof}[Proof of Theorem \ref{theolowerbound}]
If $n=1$ or $\ell=0$ then $u\in W^{1,2}(\omega_h;\R^2)$ and the assertion follows from  Lemma \ref{lemmalowerbnofracture}. Therefore we can assume $n\ge 2$, $\ell>0$. 

By \eqref{eqcalhjcjn} and Lemma \ref{lemmalowerbound2} we have
\begin{equation}
\energy_h[u]\ge c_J n\gamma \ell + 
  c\min\left\{\frac{\alpha^2h^3}{L},
\frac{\alpha^2h^3}{n^2(\ell+h/n)}\right\},
\end{equation}
where $\ell:=x_+-x_-\in(0,L]$, $n\in[1,N]$. 
As $1/(a+b)\ge\min\{1/2a,1/2b\}$ for $a,b>0$ we have, with $c':=\min\{c_J, c/2\}$,
\begin{equation}
\energy_h[u]\ge c' \left[n\gamma \ell + 
  \min\left\{\frac{\alpha^2h^3}{L},
\frac{\alpha^2h^3}{\ell n^2},
\frac{\alpha^2h^2}{n}\right\}\right],
\end{equation}
which implies
\begin{equation}\label{eqjuthreeterms}
\energy_h[u]\ge 
  c'\min\left\{\frac{\alpha^2h^3}{L},
 n\gamma \ell + \frac{\alpha^2h^3}{\ell n^2},
\frac{\alpha^2h^2}{n}\right\}.
\end{equation}
We treat the three terms separately. For the first one there is nothing to do, for the last one we simply use $n\le N$. For the middle one we use two estimates. From 
$n\le N$ and $\ell\le L$, we get
\begin{equation}
 n\gamma \ell + \frac{\alpha^2h^3}{\ell n^2}\ge 
 \frac{\alpha^2h^3}{LN^2}
\end{equation}
and  using $ab\ge2\sqrt{ab}$ first, and then $n\le N$,
\begin{equation}\label{eqlboptell}
 n\gamma \ell + \frac{\alpha^2h^3}{\ell n^2}\ge 2
  \frac{\alpha h^{3/2}\gamma^{1/2} }{n^{1/2}}\ge 2
  \frac{\alpha h^{3/2}\gamma^{1/2} }{N^{1/2}}.
\end{equation}
Therefore
\begin{equation}
 n\gamma \ell + \frac{\alpha^2h^3}{\ell n^2}\ge \max\left\{
  \frac{\alpha h^{3/2}\gamma^{1/2} }{N^{1/2}},
  \frac{\alpha^2h^3}{LN^2}\right\} \ge \frac12 
\left(  \frac{\alpha h^{3/2}\gamma^{1/2} }{N^{1/2}}+
  \frac{\alpha^2h^3}{LN^2}\right),
  \end{equation}
and inserting this bound in \eqref{eqjuthreeterms} yields
\begin{equation}
\energy_h[u]\ge 
  \frac12 c'\min\left\{\frac{\alpha^2h^3}{L},
 \frac{\alpha \gamma^{1/2}  h^{3/2}}{N^{1/2}}
 + \frac{\alpha^2h^3}{LN^2},
\frac{\alpha^2h^2}{N}\right\}
\end{equation}
which concludes the proof.
\end{proof}

\begin{remark}
We remark that the first inequality in \eqref{eqlboptell} 
corresponds to the choice $\ell:=\alpha h^{3/2}/(\gamma^{1/2}n^{3/2})$. Therefore in the regime in which the energy scales as $\frac{\alpha \gamma^{1/2} h^{3/2} }{N^{1/2}}$, 
delamination necessarily occurs over a length which is (up to a factor) equal to $\ell:=\alpha h^{3/2}/(\gamma^{1/2}n^{3/2})$. This implies localization of delemination to a small part of the sample, in agreement with the upper bounds discussed above and the experimental observations in Figure~\ref{fig:hinge_exp}.

Analogously, the lower bound shows that bending localizes. Indeed, from \eqref{eqrp0r0} and the corresponding equation we obtain
\begin{equation}
\frac{h^3}{L}\left[ |R_+^0-R_{-\alpha}|^2+
 |R_-^0-R_{\alpha}|^2\right]\le c
 \energy_h[u].
\end{equation}
We recall that $R_+^0-R_{-\alpha}$ and 
$R_-^0-R_{\alpha}$ are the differences in rotation inside the two non-delaminated parts of the sample. 
Therefore if the energy is significantly smaller than $c\alpha^2h^3/L$ the two non-delaminated parts of the sample carry a small part of the total rotation, and bending localizes to the delaminated part.
\end{remark}

\section{Discussion} \label{sec:discussion}
\begin{figure}
\begin{center}
\includegraphics[width=0.45\linewidth]{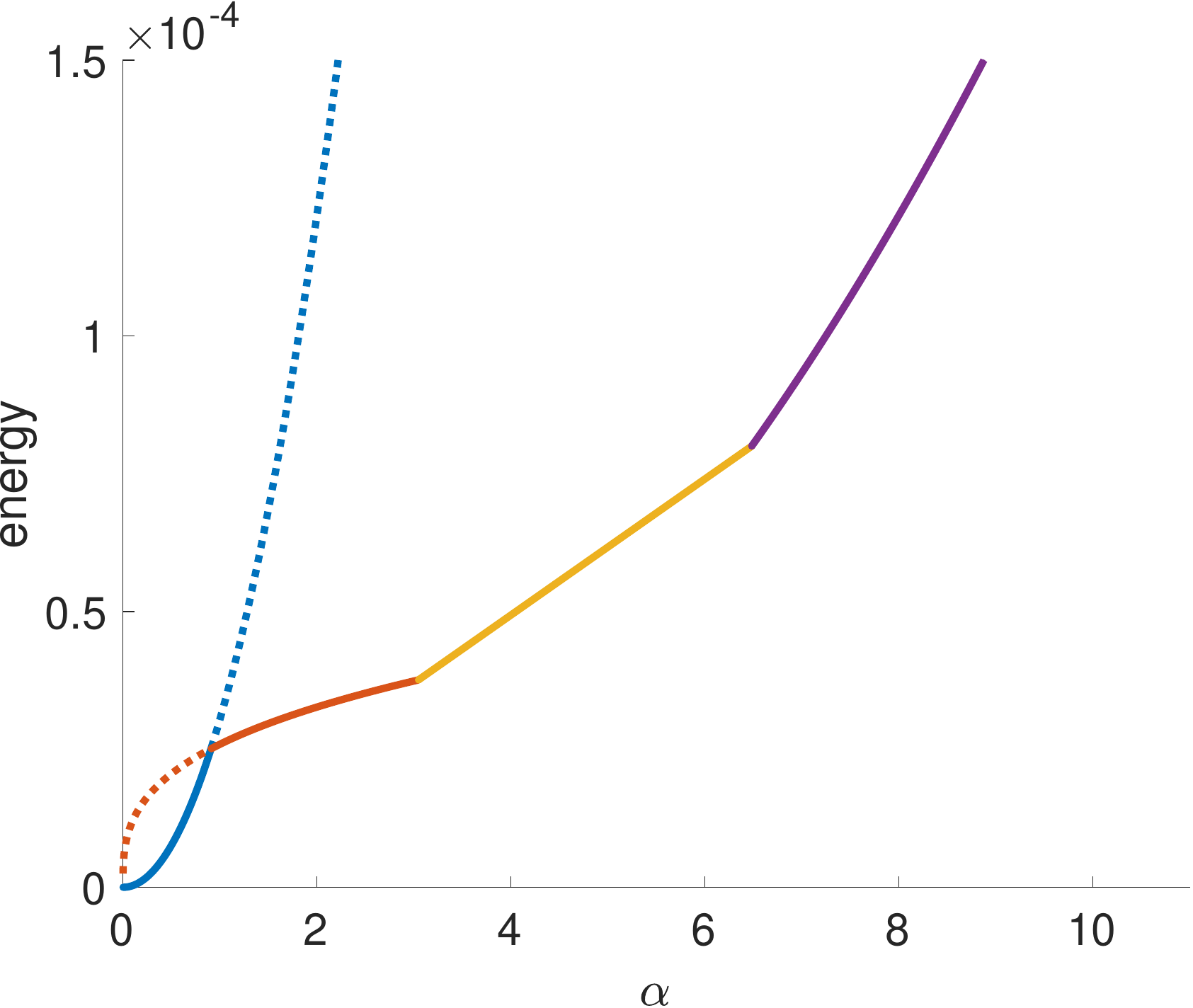}
\hskip 0.09\linewidth
\includegraphics[width=0.45\linewidth]{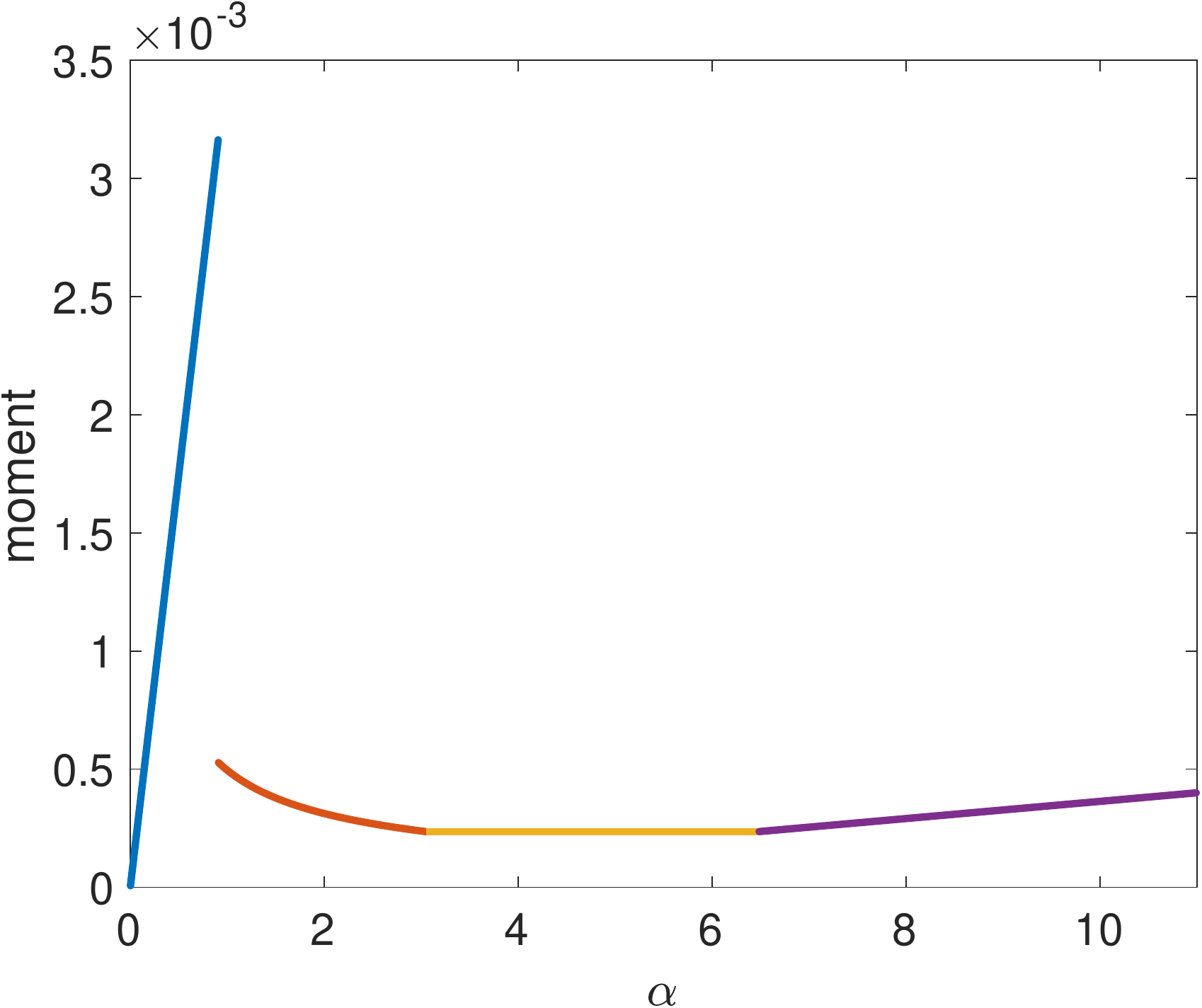} 
\end{center}
\caption{Left: plot of the upper bound of the energy from Theorem \ref{thm:scaling_upper}, depending on the bending angle $\alpha$, in the regime of \eqref{eqecasessortedinallpha}. Right: plot of the moment (i.e., the derivative of the energy) depending on the bending angle $\alpha$. In both figures, the colors indicate the respective regime (blue: non-delaminated, red: localized full delamination. yellow: localized full delamination redux, purple: total delamination). The jump in the right picture is at the position where delamination becomes energetically favorable, i.e., where the red and blue lines cross on the left and the red becomes solid and the blue becomes dotted. The parameters for the plots displayed here are $h=1$, $L=10$, $N=8$, $\gamma=10^{-6}$.} \label{fig:energy_moment}
\end{figure}

We have devised a variational model for the study of the delamination of paperboard undergoing bending. As illustrated in Figure \ref{fig:energy_moment} (left), a rich variety of energy scaling regimes emerges (non-delaminated, locally delaminated, and totally delaminated), where the locally delaminated regime exhibits the inelastic hinge found in experiments. In the case of a small coefficient $\gamma$ for the delamination energy, our model exhibits total delamination, where each layer deforms independently (apart from injectivity constraints). The small-angle total delamination regime is not included in the figure. This regime would replace the full delamination regime with linear growth, but only occurs for extremely small $\gamma$.

The bending moment is displayed in Figure \ref{fig:energy_moment}. We note that the moment should be continuous in the bending angle $\alpha$, apart from the point of the first order phase transition at the onset of delamination. This figure should be compared to \cite[Figure 18]{BP09} -- the initial undelaminated elastic response, the flat regime of increasing delamination, and the final increase of the moment are all displayed there -- although the final increase seems to be of a different origin in the experiments (where it is an artefact of the set-up) than in our variational model (where it is due to total delamination, which may only occur for unphysically small $\gamma$). The moment discontinuity we observe in our variational model at the onset of delamination, however, seems to be prevented by the introduction of an initial crease in the experiments in \cite{BP09}.

Future work should include a more detailed study of the paperboard material, taking into account the effect of its constituent cellulose fibers. 
This could involve also a treatment via the theory of homogenization, which was applied to the study of strength and fatigue of materials for example in \cite{orlik2005homogenization}. From a mathematical perspective, some open problems remain for the lower bound energy estimate: currently, it is not completely ansatz-free, as we make assumptions on the delamination sets. Furthermore, the lower bound corresponding to the partially delaminated regime is not yet available.

\section*{Acknowledgments}
This work was partially supported
by the Deutsche Forschungsgemeinschaft through projects 211504053/SFB1060, 441211072/SPP2256, and 441523275/SPP2256. PD and JO are grateful for the hospitality afforded by the Institute for Applied Mathematics of the University of Bonn.

\bibliography{biblio}

\newcommand{\etalchar}[1]{$^{#1}$}
\begin{thebibliography}{BCDM00}

\bibitem[AFP00]{AmbrosioFuscoPallara}
L.~Ambrosio, N.~Fusco, and D.~Pallara.
\newblock {\em Functions of bounded variation and free discontinuity problems}.
\newblock Oxford Mathematical Monographs. The Clarendon Press Oxford University
  Press, New York, 2000.

\bibitem[Ant95]{Antman95}
S.~S. Antman.
\newblock {\em Nonlinear problems in elasticity}.
\newblock Number 107 in Applied Math. Sciences. Springer-Verlag, 1995.

\bibitem[AP10]{AudolyPomeau2010Buch}
B.~Audoly and Y.~Pomeau.
\newblock {\em Elasticity and geometry: from hair curls to the non-linear
  response of shells}.
\newblock {O}xford {U}niversity {P}ress, 2010.

\bibitem[BCDM00]{BCDM00}
H.~{Ben Belgacem}, S.~Conti, A.~DeSimone, and S.~M{\"u}ller.
\newblock {Rigorous bounds for the {F{\"o}ppl-von K{\'a}rm{\'a}n} theory of
  isotropically compressed plates}.
\newblock {\em J. Nonlinear Sci.}, 10:661--683, 2000.

\bibitem[BCM17]{BourneContiMueller2017}
D.~Bourne, S.~Conti, and S.~M{\"u}ller.
\newblock Energy bounds for a compressed elastic film on a substrate.
\newblock {\em J. Nonlinear Science}, 27:453--494, 2017.

\bibitem[BFM08]{bou-fra-mar}
B.~Bourdin, G.~A. Francfort, and J.-J. Marigo.
\newblock The variational approach to fracture.
\newblock {\em J. Elasticity}, 91:5--148, 2008.

\bibitem[BP09]{BP09}
L.~A.~A. Beex and R.~H.~J. Peerlings.
\newblock An experimental and computational study of laminated paperboard
  creasing and folding.
\newblock {\em International Journal of Solids and Structures}, 46:4192--4207,
  2009.

\bibitem[BP12]{BeexPeerlings2012}
L.~A. Beex and R.~H. Peerlings.
\newblock On the influence of delamination on laminated paperboard creasing and
  folding.
\newblock {\em Philosophical Transactions of the Royal Society A: Mathematical,
  Physical and Engineering Sciences}, 370:1912--1924, 2012.

\bibitem[CM08]{ContiMaggi2008}
S.~Conti and F.~Maggi.
\newblock Confining thin elastic sheets and folding paper.
\newblock {\em Arch. Rat. Mech. Anal.}, 187:1--48, 2008.

\bibitem[FJM02]{FrieseckeJamesMueller2002}
G.~Friesecke, R.~D. James, and S.~M{\"u}ller.
\newblock A theorem on geometric rigidity and the derivation of nonlinear plate
  theory from three-dimensional elasticity.
\newblock {\em Comm. Pure Appl. Math.}, 55:1461--1506, 2002.

\bibitem[Hua11]{Hua11}
H.~Huang.
\newblock {\em Numerical and experimental investigation of paperboard creasing
  and folding}.
\newblock PhD thesis, KTH Royal Institute of Technology, 2011.

\bibitem[JS01]{JinSternberg2}
W.~Jin and P.~Sternberg.
\newblock {Energy estimates of the {von K{\'a}rm{\'a}n} model of thin-film
  blistering}.
\newblock {\em J. Math. Phys.}, 42:192--199, 2001.

\bibitem[KM94]{KM94}
R.~V. Kohn and S.~M\"uller.
\newblock Surface energy and microstructure in coherent phase transitions.
\newblock {\em Comm. Pure Appl. Math.}, 47:405--435, 1994.

\bibitem[Law93]{lawn1993fracture}
B.~Lawn.
\newblock {\em Fracture of brittle solids}.
\newblock {C}ambridge {U}niversity {P}ress, 1993.

\bibitem[LGL{\etalchar{+}}95]{Lobkovsky95}
A.~E. Lobkovsky, S.~Gentges, H.~Li, D.~Morse, and T.~A. Witten.
\newblock Scaling properties of stretching ridges in a crumpled elastic sheet.
\newblock {\em Science}, 270:1482--1485, 1995.

\bibitem[LM09]{Lecumberry.2009}
M.~Lecumberry and S.~Müller.
\newblock {Stability of Slender Bodies under Compression and Validity of the
  von Kármán Theory}.
\newblock {\em Archive for Rational Mechanics and Analysis}, 193:255--310,
  2009.

\bibitem[LR95]{LeDretRaoult95}
H.~LeDret and A.~Raoult.
\newblock The nonlinear membrane model as a variational limit of nonlinear
  three-dimensional elasticity.
\newblock {\em J. Math. Pures Appl.}, 73:549--578, 1995.

\bibitem[LW{\"O}17]{LWO17}
E.~Linvill, M.~Wallmeier, and S.~{\"O}stlund.
\newblock A constitutive model for paperboard including wrinkle prediction and
  post-wrinkle behavior applied to deep drawing.
\newblock {\em International Journal of Solids and Structures}, 117:143--158,
  2017.

\bibitem[Orl05]{orlik2005homogenization}
J.~Orlik.
\newblock Homogenization of strength, fatigue and creep durability of
  composites with near periodic structure.
\newblock {\em Mathematical Models and Methods in Applied Sciences},
  15:1329--1347, 2005.

\bibitem[{\"O}st17]{Ost17}
S.~{\"O}stlund.
\newblock Three-dimensional deformation and damage mechanisms in forming of
  advanced structures in paper.
\newblock In {\em Advances in pulp and paper research, Transactions of the 16th
  Fundamental Research Symposium held at Oxford}, pages 489--594, 2017.

\bibitem[Sim20]{Sim20}
J.-W. Simon.
\newblock A review of recent trends and challenges in computational modeling of
  paper and paperboard at different scales.
\newblock {\em Archives of Computational Methods in Engineering}, pages 1--20,
  2020.

\bibitem[Ven04]{Venkataramani2004}
S.~C. Venkataramani.
\newblock Lower bounds for the energy in a crumpled elastic sheet -- a minimal
  ridge.
\newblock {\em Nonlinearity}, 17:301--312, 2004.

\end{thebibliography}
\bibliographystyle{alpha-noname}

\end{document}